\documentclass[11pt]{article}
\usepackage{amsmath,amsthm,amsfonts,amssymb,mathrsfs,epsfig,bm}
\usepackage{color}

\parskip 	\smallskipamount

\newtheorem{theorem}{Theorem}
\newtheorem{lemma}{Lemma}
\newtheorem{corollary}{Corollary}
\newtheorem{proposition}{Proposition}

\newtheorem{remark}{Remark}

\def\Z{\mathbb{Z}}

\def\R{\mathbb{R}}
\def\C{\mathbb{C}}
\def\FF{\mathscr{F}}
\def\BB{\mathscr{B}}
\def\RR{\mathcal{R}}

\def\NN{\mathcal{N}}
\def\WW{\mathsf{W}}
\def\ZZ{\mathsf{Z}}

\renewcommand{\phi}{\varphi}
\renewcommand{\epsilon}{\varepsilon}

\newcommand{\1}{{\text{\Large $\mathfrak 1$}}}
\newcommand{\comp}{\raisebox{0.1ex}{\scriptsize $\circ$}}
\newcommand{\cadlag}{{c\`adl\`ag} }

\renewcommand{\liminf}{\varliminf}

\newcommand{\eqdist}{\stackrel{\text{\rm d}}{=}}

\newcommand{\ee}{\mathfrak e}

\begin{document}

\title{\bf Analysis of stochastic fluid queues driven by local time processes}
\author{
{\sc Takis Konstantopoulos}\thanks{School of Mathematical and Computer
Sciences, Heriot-Watt University, Edinburgh EH14 4AS, UK} 
%\hspace*{1cm} 
\and
{\sc Andreas E. Kyprianou}\thanks{Department of Mathematical Sciences,
The University of Bath, Bath BA2 7AY, UK}
\and 
{\sc Paavo Salminen}$^\ddagger$$^0$
%\hspace*{1cm} 
\and
{\sc Marina Sirvi\"o}\thanks{Department of Mathematics, \AA bo Akademi
University, Turku, FIN-20500, Finland}
}
\maketitle
\date{}
\footnotetext{Corresponding author}

\begin{abstract}
We consider a stochastic fluid queue served by a constant rate server
and driven by a process which is the local time of a certain Markov process.
Such a stochastic system can be used as a model in a priority service system,
especially when the time scales involved are fast. 
The input (local time) in our model is always singular with
respect to the Lebesgue measure which in many applications is 
``close'' to reality.
We first discuss how to rigorously construct the (necessarily) unique 
stationary version of the system under some natural stability conditions.
We then consider the distribution of
performance steady-state characteristics, namely, the 
buffer content, the idle period and the busy period.
These derivations are much based on the
fact that the inverse of the local time of a Markov process is a L\'evy process (a subordinator) hence making 
the 
%powerful 
theory of L\'evy processes applicable. Another important ingredient in our approach 
is the Palm calculus coming from the point process point of view.
%We take advantage of the fact that the input (local time) is singular with
%respect to the Lebesgue measure and use a Palm calculus technique to
%shed light to distributional relations. We also take advantage of the
%fact that the inverse of the local time is a L\'evy process (a subordinator)
%to explicitly derive various distributions.
\\ \\%\bigskip\noindent
\\ \\ %\bigskip\noindent
%
%\vspace*{2mm}
%\ams{???}{???}
%% key to ams subject classification:
%60G10 Stationary processes
%60G50 Sums of independent random variables; random walks
%60G51 Processes with independent increments
%90B15 Network models, stochastic
\end{abstract}

\section{Introduction}
\label{sec1}
This paper extends the results of Mannersalo et al.\ \cite{MNS} who
introduced a fluid queue (or storage process) driven by the local
time at zero of a reflected Brownian motion and served by a deterministic
server with constant rate. 
The motivation provided in \cite{MNS} is that the system provides a macroscopic
view of a priority  queue with two priority classes.  Indeed, in such
a system, the highest priority class (class 1) goes through as if the lowest
one does not exist, whereas the lowest priority class (class 2) gets served
whenever no item of the highest priority is present. In telecommunications
terminology, class 2 only  receives whatever bandwidth
remains after class 1 served. As argued in \cite{MNS},
if  the highest priority queue is,  macroscopically, approximated by
a reflected Brownian motion, the lowest priority queue is driven
by the cumulative idle  time  of the first one, which is approximated by
the local time of the reflected Brownian motion at $0$.

In view of Internet networking applications, such as the service provision
amongst several classes of service (e.g.\ streaming video and expedited data),
the fluid or macroscopic model is thus quite appropriate for obtaining
a better picture of the situation and for performance analysis and design.
{From} a mathematical point of view, the model is a  rare example  
of a non-trivial fluid  queue  whose performance characteristics  (such as
steady-distribution) can be computed explicitly.
If, in addition, we take into account the heavy-tailed nature of traffic
on the Internet, it seems reasonable to consider a L\'evy process as
a model for class 1 queue. This provides motivation for studying 
a queue whose input is the local time of a reflected L\'evy process.

More generally, let $X$ be a Markov process and $L$ its local time
at a specific point. 
The fluid queue driven by $X$ refers to the
stochastic system defined by
\[
Q_t = Q_0 + L_t - t + I_t, \quad t \ge  0,
\]
where $Q_t \ge 0$ for all $t \ge 0$, and
$I$ is a non-decreasing process, starting from $0$, such that 
\[
\int_0^\infty \1(Q_s > 0) dI_s =0.
\]
Thus, $Q$ is obtained by Skorokhod reflection and $I$ is necessarily 
given by 
$$
I_t = -\inf_{0  \le s \le t} [(Q_0+L_s-s)\wedge 0];
$$
see \cite{KL}.
By considering, instead of $0$, an arbitrary initial time, we can
define a proper stochastic dynamical system (see Appendix A for details)
which, under natural conditions, admits a unique stationary version.
To this we refer frequently throughout the paper. 

We remark also that in 
Kozlova and Salminen \cite{KoSa} the situation in which $X$ is a general one-dimensional diffusion 
is analysed. Moreover, Sirvi\"o (n\'ee Kozlova) \cite{Si} studies the case where $L$ is constructed as the inverse 
of a general subordinator (without specialising the underlying process $X$). 

This paper follows ideas which were developed in \cite{Si} in the context of
reflection of the inverse of a subordinator. However, 
(1) it connects
the abstract framework with the case where the subordinator is the
local time of a reflected L\'evy process (motivated by applications in
priority processing systems)
(2) it uses, as much as possible, a framework based on Palm probabilities
(see, in particular, Section 4--Theorem \ref{distr0} and Section 5--Lemma 
\ref{pikk2}), and 
(3) explicitly discusses the possible types of sample path behaviour
of the process of interest ($Q$); for illustration, see Figures 1, 2 and 3:
Figure 1 concerns the case where $Q$ has continuous sample paths but
with parts which are singular with respect to
the Lebesgue measure.
Figure 2 concerns the case where $Q$ has paths with isolated 
discontinuities (positive jumps) and linear decrease between them.
Figure 3 concerns the case where $Q$ has absolutely continuous paths.
The cases are exhaustive.
Despite the wide variety of sample paths (depending on the type
of underlying L\'evy process $Y$), the mathematical framework and
formulae derived have a uniform appearance.

The paper is organised as follows. In Section \ref{sec2} we construct
the stationary version of the underlying (background) Markov process $X$.
In Section \ref{sec3} we construct
the stationary version of stochastic fluid queue with input the local
time of $X$, based on the stationary version of $X$.
In Section \ref{stationary-section} we derive the stationary distribution of the
buffer content and present a number of examples.
In Section \ref{sec5} we examine the idle and busy periods, and, in particular, characterise the distributions 
of their starting and ending times. 
The analysis is carried out first under the condition 
that the time at which the system is observed is a typical point of and
idle or a busy period.
Finally, the distributions of typical idle and busy periods are derived. 
%We then make some observations regarding two methods to describe
%the system via Palm calculus.

\section{The background Markov process and its local time}
\label{sec2}
We first construct the underlying Markov process $X$ which models
the highest priority class. This process will be taken to be the stationary
reflection of a spectrally one-sided L\'evy process
\[
Y = (Y_t, ~ t \in \R)
\]
with two-sided time and $Y_0=0$ (see Appendix B). A L\'evy process is called 
spectrally negative if its L\'evy measure $\Pi$ satisfies
$$
\Pi((-\infty,0))>0 \quad {\rm and}\quad   \Pi((0,+\infty))=0,
$$
and spectrally positive if
$$
\Pi((-\infty,0))=0 \quad {\rm and}\quad   \Pi((0,+\infty))>0.
$$
Clearly, if $Y$ is 
spectrally positive then $-Y$ is spectrally negative, and vice versa. To avoid 
trivialities we shall throughout assume that 
\[
\text{\em $|Y|$ does not have monotone paths}
\]
which rules out the cases that $Y$ is an increasing or decreasing subordinator.

We also
discuss the characteristics of its local time at $0$.
Appendix A summarises the notation and results on the Skorokhod reflection
problem and its stationary solution.
Appendix B summarises some facts on L\'evy processes with one sided jumps indexed by $\mathbb{R}$.
We will throughout denote by $P$ a probability measure which
is invariant under time shifts, and by $P_x$ the conditional probability
measure when $X_0=x$.
We define the Laplace exponent of a spectrally one-sided L\'evy process
as a function $\psi_Y : \R_+ \to \R$ given for $\theta\geq 0$ by
\[
\psi_Y(\theta) :=
\begin{cases}
\log E e^{\theta(Y_{t+1}-Y_t)}, &
\text{ if $Y$ is spectrally negative },\\
\log E e^{-\theta(Y_{t+1}-Y_t)}, &
\text{ if $Y$ is spectrally positive}.
\end{cases}
%\quad
%\theta \ge 0.
\]
Thus we insist that $\psi_Y$ be defined on $\R_+$, and define its right inverse 
\begin{equation}
\label{e1}
\Phi_Y(q) := \sup\{\theta \ge 0:~ \psi_Y(\theta) = q\}, \quad q \ge 0.
\end{equation}
We use also the notation 
$$
\overline{Y}_t:=\sup_{0\leq s\leq t} Y_s\quad{\rm and}\quad 
\underline{Y}_t:=\inf_{0\leq s\leq t} Y_s, 
\quad t \ge 0,
$$
and recall the duality lemma for L\'evy processes (see, e.g., Bertoin \cite[p.\ 45]{BER}):
\begin{equation}
\label{e2}
\{Y_t-Y_{(t-s)-}\,:\, 0\leq s\leq t\}  \eqdist \{Y_{s}\,:\, 0\leq s\leq t\},
\end{equation}
where $\eqdist$ means equality in distribution. Hence,
$$
\sup_{0\leq s\leq t}\left(Y_t-Y_{(t-s)-}\right)
\eqdist
\overline{Y}_t,%\sup_{0\leq s\leq t}\left(Y_s\right)
$$
which is equivalent with 
\begin{equation}
\label{e3}
Y_t -\underline{Y}_t
\eqdist
\overline{Y}_t.
%\sup_{0\leq s\leq t} Y_s.
\end{equation}
In this paper, we will mainly study the L\'evy process $Y$ with the time parameter taking values in the 
whole of $\R$ (see Appendix B). 
The Skorokhod reflection mapping associated with 
$\{Y_t\,:\, t\in\R\}$ is defined (see Lemma \ref{u1} in Appendix A)
via
$$
X_t := \widetilde \RR_t Y 
:= \sup_{-\infty < s \le t} (Y_t-Y_s), \quad t \in \R.
$$
In the remaining of this section we give conditions for the existence
of the stationary process\footnote{That this process is Markov is easy
to see due to the independence of the increments of $Y$.}
$\widetilde \RR Y=(\widetilde \RR_t Y, t \in \R)$, 
compute its marginal distribution, and define the local time process
$L$ of $\widetilde \RR Y$ at 0 which will be used for the construction of the fluid queue.
A few words about the definition of $L$ are in order.
We adopt the point of view the $L$ is a stationary
random measure on $(\R, \BB)$,
i.e.\
\[
L(s,s+t] = L(0,t] \comp \theta_s,
\quad t \ge 0, \quad s \in \R,
\]
where $\theta_s$ is the shift on the canonical space (see Appendix B),
which regenerates together with $X$ at each point $t$ at which $X_t=0$.
It is known that $L$ is a.s.\ continuous if and only if the point $x=0$
is regular for the closed interval
$(-\infty, 0]$ for the process $X$, and this is equivalent
to $\inf\{t > 0:~ Y_t \le 0\} =0$, $P_0$-a.s.
Furthermore, $L$ is a.s.\ absolutely continuous if and only if, in addition
to the above, the point $x=0$ is irregular for the open interval $(0,\infty)$
for the process $X$, and this is equivalent to 
$\inf\{t > 0:~ Y_t > 0\} >0$, $P_0$-a.s.
If $L$ is a.s.\ continuous then it is not difficult to attach a physical
meaning to it as a cumulative input process to a secondary queue.
For mathematical completeness, we shall also consider the case where
$L$ is a.s.\ discontinuous, in which case it can be shown to
have a discrete support. In the continuous case, $L$ can be defined
uniquely module a multiplicative constant. We shall make the normalization precise later. 
In the discontinuous case,
there is more freedom; however, insisting that its inverse be a subordinator,
we are left with only one choice. The discontinuous case
appears only once below and the construction of $L$ is discussed there.
In all cases, the support of the measure $L$ is the closure of the set
$\{t\in \R : X_t=0\}$.

Associated to the measure $(L(B), B \in \BB)$ we 
can define a cumulative local time process, denoted (abusing notation), by
the same letter, and given by:
\[
L_t :=
\begin{cases}
 L(0,t],&  t \ge 0,\\
-L(t,0], & t < 0.
\end{cases}
\]
The right-continuous inverse process is
\begin{equation}
\label{linv}
L^{-1}_x := 
\begin{cases}
\inf\{ t > 0:~ L(0,t] > x\} , & x \ge 0 \\
\sup\{ t < 0:~ L(t,0] < x\} , & x <  0.
\end{cases}
\end{equation}
In case $L$ is $P$-a.s.\ continuous, the process $(L^{-1}_x, x \in \R)$ has,
under $P_0$,
independent increments and a.s.\ increasing paths
(i.e.\ it is a, possibly killed, subordinator). 
This is an additional requirement that needs to be
imposed when $L$ is not $P$-a.s.\ continuous.

Throughout the paper, we denote by $P_x$ the probability $P$,
conditional on $X_0=x$.

\subsection{Stationary reflection of a spectrally negative L\'evy process}
\label{sn}
Suppose that the process $Y$ is a spectrally negative L\'evy process
with non-monotone paths; see expression \eqref{levyito}. 
%Let $\Pi$ be its
%L\'evy measure and $\psi_Y(\theta) = \log E e^{\theta Y_1}$, $\theta >0$,
%its characteristic exponent.
%\begin{equation}
%\label{ppp}
%\Pi(-\infty, 0) > 0.
%\end{equation}
%Note that $E_0 Y_1 = \psi_Y'(0+) \in [-\infty, +\infty)$.
%Consider the Skorokhod dynamical system
%$\RR_{s,t}^Y$, driven by $Y$.
%Directly from \eqref{sds} we have: 
%\begin{align}
%\RR_{0,t}^Y(0) &= Y_t - \underline{Y}_t 
%= Y_t - \inf_{0\le s \le t} Y_s.
%\label{ydown}
%\end{align}
%The following concerns the stationary version of the 
%reflected process.
\begin{proposition}
\label{stationary-distn}
Let  $Y=\{Y_t\,:\, t \in \R\}$ be a spectrally negative L\'evy process with
two-sided time. Assume that its Laplace exponent 
$\psi_Y(\theta) = \log E e^{\theta Y_1}$, $\theta >0$,  is
such that $\psi_Y'(0+) < 0$. Then 
$$
X=\{
X_t := \widetilde \RR_t Y\,:\, t \in \R
\}
$$
is the unique stationary solution of SDS (the Skorokhod dynamical system, 
see Appendix A) driven by $Y$.
The marginal distribution of $X$ is exponential with mean $1/\Phi_Y(0).$ 
\end{proposition}
\proof
Since $E [Y_{t+1}-Y_t] = \psi'_Y(0+)<0$,
existence and uniqueness of the stationary solution 
is guaranteed by 
%Lemma \ref{u1} 
Corollary \ref{u2} of Appendix \ref{skororeview}. 
That $\Phi_Y(0)>0$ is a direct consequence of the definition of $\Phi_Y$
(see \eqref{e1}).
%of the function $\Phi_Y,$  . 
To derive the marginal distribution of $X$ consider
for $ \beta \ge 0$
\begin{eqnarray}
\label{stat_dist}
&&
\nonumber
E[e^{-\beta X_0}] 
= \lim_{t\to \infty} E_0[e^{-\beta (Y_t-\underline{Y}_t)}]
=\lim_{t\to \infty} E_0[e^{-\beta \overline{Y}_t}]
\end{eqnarray}
Since  $Y$  is spectrally negative its over all supremum, 
$\sup_{s\geq 0}Y_s,$ is exponentially distributed with mean $1/\Phi_Y(0)$
(see, e.g., Bertoin \cite[p.\ 190]{BER} p. 190 or Kyprianou  \cite[p.\ 85]{K}). 
\qed

%Note that the process \eqref{ydown} is Markovian. 
%In fact, $(X_t, t \in \R)$
%is the (2-sided) stationary version of $(\RR_{0,t}^Y(x), t \ge 0)$, and
%\[
%(X_t, t  \ge 0) \eqdist (\RR_{0,t}^Y(x), t \ge 0), 
%\quad \text{under $P_x$, for all $x \ge 0$}.
%\]

%We define the local time $L$ of $X=(X_t, t \in \R)$ at the point $x=0$
%as an additive functional with nonnegative increments; 
%if we denote the increment on $(s,t]$ by $L(s,t]$, stationarity means 
%\[
%L(s,s+t] = L(0,t] \comp \theta_s, \quad t \ge 0, \quad s \in \R,
%\]
%where $\theta_t$ is defined on the canonical space (see Appendix B).
%It is only the increments that make sense. We may arbitrarily
%set $L_0=0$ and define (with a slight abuse of notation)
%the process $(L_t, t \in \R)$
%by 
%\[
%L_t =
%\begin{cases}
 %L(0,t],&  t \ge 0,\\
%-L(t,0], & t < 0,
%\end{cases}
%\]
%so that $L(s,t] = L_t-L_s$ for all $-\infty < s \le t < \infty$.
%
%The process $(L_t, t \in \R)$ does not necessarily have continuous sample paths.

In view of Proposition \ref{stationary-distn}, we assume that
\[
\psi'_Y(0+) \in [-\infty, 0),
\]
which is equivalent to 
\[
\Phi_Y(0) > 0.
\]

It is easily seen that, for a non-monotone spectrally negative $Y$,
a necessary and sufficient condition 
for continuity of $L$ is that $Y$ has unbounded variation paths.
This is further equivalent to: 
$\sigma >0$ or $\int_{-1}^0 |y| \Pi(dy) = \infty$. 

%Moreover, in this case, $L_t=-\underline{Y}_t$ if $X_0=0.$  {\color{red} ?} ABSOLUTELY NOT!

In the alternative case, when the paths of $Y$ are of bounded variation,
the number of visits of $Y$ to its running infimum forms
a discrete set. So $n(s,t] := \sum_{u \in \R} \1(X_u=0)$ is
finite for all $-\infty <s < t<\infty$. Let $n_t := n(0,t]$ if $t \ge 0$,
and $n_t :=-n(t,0]$ if $t < 0$, and let $(\ee_j, j \in \Z)$ be a collection
of i.i.d.\ exponentials with mean $1$, independent of $Y$.
We adopt the following construction for $L$.
%\marginpar{\tiny\sc locally finite? no way!}
%In this case, we can construct a local time with \cadlag paths that
%regenerates at the points of this discrete set. The
%standard construction takes the form
\[
L(s,t] = \sum_{n_s < i \le n_t} \ee_i, \quad -\infty <s \le t<\infty.
\]

In both cases, the process 
$(L^{-1}_x, x \in \R)$ is a subordinator under $P_0$, with
$L^{-1}_0 =0$, $P_0$-a.s.
If $L$ is continuous, this property is immediate from the definition of $L$.
If $L$ is discontinuous, $L^{-1}$ has independent increments due to our choice
of the exponential jumps of $L$.
Since $\psi'_Y(0+) < 0$, we have $Y_t\to \mp \infty$, 
as $t \to \pm \infty$, $P$-a.s., and this implies that $L_t \to \pm \infty$,
as $t \to \pm \infty$, $P$-a.s. Thus, it is not possible for $L_x^{-1}$ to
explode for finite $x$.

%A fundamental property of $(L_t, t \ge 0)$ is that its inverse process
%\begin{equation}
%\label{inverseloc}
%L^{-1}_x:= \inf\{t\geq 0: L_t>x\}, \quad x \ge 0,
%\end{equation}
%is a pure jump subordinator (possibly killed at an independent exponentially
%distributed time if the point zero is transient for the underlying
%Markov process).  
Regardless of the continuity of the paths of $L$, we always have the following:

\begin{proposition}%{lemma}
\label{loc-}
Let $Y$ and $X$ be as in Proposition \ref{stationary-distn} and % spectrally negative with nontrivial L\'evy measure, and
$L$ the local time at $0$ of $X.$%$\{X_t := \widetilde \RR_t Y\,:\,t \in \R\}$.
Then the local time $L$ can be normalized to satisfy  
\begin{equation}
\label{ltau0}
E_0 \big[ e^{-q L^{-1}_x} \big] = e^{- x q/\Phi_Y(q)},
\quad q \ge 0,
\end{equation}
and, moreover, 
\begin{equation}
E_0 L_x^{-1} = x/{\Phi_Y(0)}, \quad x \in \R, \qquad
E L_t=t\,\Phi_Y(0), \quad t \in \R.
\label{lstat1}
\end{equation}
%\in (0, \infty).$
%{\color{red} have omitted the condition $\Phi_Y(0) > 0$ since this is implied  by
%$\psi_Y'(0+) < 0$ ? --YES!}
% \frac{x}{\Phi_Y(0)}
%Note that the quantity $\Phi_Y(0)>0$ because $\psi_Y(0+)<0$.
\end{proposition}%{lemma}
\proof
For the Laplace exponent of  $L^{-1}_x$ in (\ref{ltau0}), we refer to
Bingham \cite[p.\ 731]{BIN} 
(a result due to Fristedt \cite{Fri}) 
and Kyprianou \cite{K} and Kyprianou and Palmowski \cite{KP}:
The ``ladder process'' $(L_x^{-1}, X_{L_x^{-1}}), x \ge 0)$
is a  L\'evy process with values in $\R^2_+$ and Laplace exponent
\[
\widehat \kappa(\alpha,\beta) = \log E_0 \big[e^{-\alpha L_1^{-1}
-\beta X_{L_1^{-1}}}\big] = \frac{\alpha-\psi_Y(\beta)}{\Phi_Y(\alpha)-\beta}
\]
obtained by Wiener-Hopf factorisation for a spectrally negative process;
see Bertoin \cite[p.\ 191, Thm.\ 4]{BER}.
Setting $\beta=0$ we obtain 
$E_0[e^{-\alpha L_1^{-1}}] = e^{-\alpha/\Phi_Y(\alpha)}$, 
as claimed.
{From} this, we obtain $E_0 L_x^{-1} = x/\Phi_Y(0)$, by differentiation.
Using the strong law of large numbers, we have $\lim_{x \to \infty}L^{-1}_x/x
= 1/\Phi_Y(0)$, $P_0$-a.s., and, given that $(L^{-1}_x, x > 0)$ 
is the right-continuous inverse function of 
$(L_t, t > 0)$--see \eqref{linv}, we have 
$\lim_{t \to \infty} L_t/t = \Phi_Y(0)$, $P_0$-a.s., and $P$-a.s.
Since $L$ is a stationary random measure, we have $EL_t=C\,t$ for some constant $C.$ Hence, we immediately have that 
$C=\Phi_Y(0)$, and this proves \eqref{lstat1}.
\qed

We shall later need the $P$-distribution of the random variable
%To compute the mean of
%$L_t$ under the stationary measure, define 
\begin{equation}
\label{D}
D:=\inf\{t>0: \, X_t=0\}.
\end{equation}
Since $X_0$ is exponential with rate $\Phi_Y(0)$, we have $P(X_0>0)=1$.
So if $0 \le t < D$, we have $X_t=X_0+Y_t$, $P$-a.s.
Therefore
\[
D = \inf\{ t > 0:~ Y_t < -X_0\}, \quad P-\text{a.s.}
\]
%Then on $\{X_0=x\},\, x>0$,
Let
$$
\tau_{-x}:=\inf\{t>0: Y_t < -x\}.
$$
  Therefore,
\begin{eqnarray} 
\label{pikk6}
\nonumber
    E[e^{-\theta D}]&=& \int_0^\infty P(X_0 \in dx)\, E[e^{-\theta
      \tau_{-x}}]\\ 
\nonumber
%&=& \int_0^\infty P(X_0 \in dx)\, E_x [e^{-\theta \tau_0}] 
%    \nonumber \\
    &=& \int_0^\infty \Phi_Y(0) \, e^{-\Phi_Y(0) x} \, 
\left(Z^{(\theta)}(x) - \frac{\theta}{\Phi_Y(\theta)}W^{(\theta)}(x)\right)
\, dx
    \nonumber \\
    &=& \Phi_Y(0) \, \frac{\psi_Y(\Phi_Y(0)) \Phi_Y(\theta) - \theta \Phi_Y(0)}
    {\Phi_Y(\theta) ~ \Phi_Y(0) ~ [\psi_Y(\Phi_Y(0)) - \theta]} 
\nonumber
\\ 
&=&
\frac{\Phi_Y(0)}{\Phi_Y(\theta)},
\end{eqnarray}
where the overshoot formula (\ref{tau-}) given in the Appendix B (see also Bingham \cite[p.\ 732]{BIN}) and the Laplace transforms 
\eqref{LT}, \eqref{Ztransform}, for the scale functions $W^{(\theta)}$, 
$Z^{(\theta)}$, were used.
%Let now $\xi$ be an exponentially distributed random variable (with parameter $\theta$) independent of $X$.
%  Then, using $(\ref{pikk6})$ and  (\ref{ltau0}), we have
%  \begin{eqnarray*}
%     E L_\xi &=& \int_0^\infty P(L_\xi > x) \, dx = \int_0^\infty P(L_x^{-1} \leq \xi)\, dx \\
%     &=& \int_0^\infty E [e^{-\theta L_x^{-1}}] \, dx = 
%     \int_0^\infty E[e^{-\theta D}]\, E_0 [e^{-\theta L_x^{-1}}]\, dx \\
%&=&
%\frac{\Phi_Y(0)}{\Phi_Y(\theta)} \int_0^\infty e^{-x
%    \theta / \Phi_Y(\theta)} dx 
%\\
%  &=& \frac{\Phi_Y(0)}{\theta}.
%\end{eqnarray*}     
%Hence, the claim follows by the uniqueness of Laplace transforms. 
%$E L_t = \Phi_Y(0)\, t.$ %For $t<0$, $E L_t = - E L_{-t} = \Phi_Y(0)\, t$.
%\end{proof}

%\begin{corollary}
%\label{snc}
%If $\Phi_Y(0) > 0$ then $E_0 L_x^{-1} = \frac{x}{\Phi_Y(0)}
%\in (0, \infty)$ and $E_0 L_t  = t \Phi_Y(0) \in (0, \infty)$.
%\end{corollary}

%\begin{remark}
%In Bingham \cite{BIN} Section 5 the setting is 
%\begin{itemize}
%\item $Z$ is a spectrally negative its Laplace exponent is $\psi$ (the sign convention is as in the definition of $Y$ above),
%\item $\widehat Z:=-Z$ is then spectrally positive, 
%\item reflection $g(\widehat Z)$ is defined as $\sup_{s\leq t}\widehat Z_s -\widehat Z_t,$ and we have 
%$$
%g(\widehat Z)_t=\sup_{s\leq t}(- Z_s) +Z_t = Z_t-\inf_{s\leq t}Z_s,
%$$ 
%and, hence, this construction coincides with $\widetilde \RR_t Z.$ 
%\end{itemize}
%\end{remark}
%In the forth coming analysis, we allow
%for the use of this local time, as an input process to a fluid queue, 
%despite the fact
%that its paths are not continuous.  

\subsection{Stationary reflection of a spectrally positive L\'evy process}
\label{sp}
We can repeat the construction in the subsection above for  a 
spectrally positive L\'evy process $Y$ with non-monotone paths.
We shall be using the formulae of Appendix B with $-Y$ in place of $Y$.

\begin{proposition}%{lemma}
\label{stationary-distn+}
Let $Y=\{Y_t\,:\, t \in \R\}$ be a spectrally positive L\'evy process with
two sided time and Laplace exponent $\psi_Y$. 
Assume that its Laplace exponent $\psi_Y(\theta) = \log E e^{-\theta Y_1}$
is such that $\psi_Y'(0+)>0$. Then 
the process $\{X_t := \widetilde \RR_t Y\,:\,   t \in \R\}$
%:= \sup_{-\infty < s \le t} (Y_t-Y_s), \quad t \in \R,\]
is the unique stationary solution to the SDS driven by $Y$.
The stationary distribution of $X$ is given for $\beta\geq 0$ by 
\begin{eqnarray}
\label{ZOLO}
&&
\nonumber
E[e^{-\beta X_0}] 
= \lim_{t\to \infty} E_0[e^{-\beta (Y_t-\underline{Y}_t)}]
=\lim_{t\to \infty} E_0[e^{-\beta \overline{Y}_t)}]
\\
&&\hskip1.65cm 
= \psi_Y'(0+) \frac{\beta}{\psi_Y(\beta)}.
\end{eqnarray}
\end{proposition}
\begin{proof}
Notice that in this case, by the assumption on $\psi_Y,$ 
$$
E [Y_{t+1}-Y_t] =- \psi'_Y(0+)<0,
$$
and, hence, $Y$ drifts to $-\infty.$
Let $Z:=-Y.$ Clearly, $Z$ is spectrally negative,
$$
X_t=Y_t-\underline{Y}_t = \overline{Z}_t-Z_t,
$$
and $Z$ drifts to $+\infty.$ The classical result
due to Zolotarev \cite{zolotarev64} (see also Bingham \cite{BIN} Proposition 5 p. 725) 
says that the stationary distribution of $X$ is as given in (\ref{ZOLO}). 
\end{proof}

We shall therefore assume that
\[
\psi'_Y(0+) \in (0, \infty).
\]
Hence $\Phi_Y = \psi_Y^{-1}$ and so 
\[
\Phi_Y'(0+) = 1/\psi_Y'(0+) \in (0, \infty).
\]

It is easily seen that, starting from $0$, the process $Y$ hits $(-\infty, 0]$
immediately, $P$-a.s., and this ensures continuity of the local time $L$.
Moreover, we may and do normalize $L$ so that  
%up to a multiplicative constant (taken to be 1), we have
\begin{equation}
L(s,t]= -\inf_{ s< u \leq t } Y(s,u].% : -\inf_{0 < s \le t} Y_s.
\label{convention}
\end{equation}
The continuity of $L$ implies that
\begin{equation}
\label{ltau}
\{L_x^{-1}\,:\, x \ge 0\}=\{\tau_{-x}\,:\, x \ge 0\},
\end{equation}
where $\tau_{-x}:=\inf\{t >0:Y_t<-x\}=\inf\{t >0:Z_t>x\}.$ %(See Appendix B).
Note that $(L_x^{-1}, x \in \R)$ is a subordinator under $P_0$, with
$L^{-1}_0=0$, $P_0$-a.s. Furthermore, since $Y_t$ drifts to $\mp\infty$
as $t \to \pm\infty$, $L^{-1}$ is proper (not killed).

Let us briefly comment on the special case where,
starting from $X_0=0$, 
the interval $(0, \infty)$ will be first visited by $X$ at 
an a.s.\ positive time.
It is known \cite[Ch.\ 7]{BER} that this occurs if and only if
$Y$ has bounded variation, i.e.\
\begin{equation}
\label{Ybv}
Y(s,t] = d_Y (t-s) + \int_{(s,t]}\int_{(0, \infty)}
y ~ \eta(du,dy),
\end{equation}
where $d_Y$ is the drift, and $\int_0^1 y~\Pi(dy) < \infty$. 
Since we exclude the case where $Y$ is monotone, we must
have $d_Y < 0$.
%\[
%Y_t = d_Y t - S_t
%\]
%where $\{S_t : t\geq 0\}$ is a pure jump subordinator and $d_Y>0$.
In that case, with (\ref{convention}) as the definition
of $L$, it is known that for all $s \le t$,
\begin{equation}
L(s,t] = |d_Y| \int_s^t \1(X_u=0) du.
\label{abs-cts}
\end{equation}

A rewording of the first part of Lemma \ref{one-sided-exit} in Appendix B  gives the first part of the following proposition.
\begin{proposition}%lemma}
\label{loc+}
Let $Y$ be as in Proposition \ref{stationary-distn+} and $L$ the local time at $0$
of $X.$ %$\{X_t := \widetilde \RR_t Y$, $t \in \R\}.$ 
Then
\begin{equation}
\label{ltau1}
E_0[e^{-q L_{x}^{-1}}] = e^{-\Phi_Y(q)x}, \quad q \ge 0.
\end{equation}
Moreover, 
\begin{equation}
E_0 L_x^{-1} = x \Phi_Y'(0+), \quad x \ge 0,
\qquad
\label{lstat2}
E L_t =
t\, \psi_Y'(0+),\quad t \ge 0.
\end{equation}
%Note that $\Phi'(0+) = 1/\psi'(0+)\in(0,\infty)$ on account of the fact that $\psi'(0+)>0$.
\end{proposition}%lemma}{lemma} 
%\begin{corollary}
%\label{spc}
%If $\psi_Y'(0+)>0$ then $E_0 L_x^{-1} = x \Phi_Y'(0+)
%\in (0, \infty)$ and $E_0 L_t = t \psi_Y'(0+) \in (0, \infty)$.
%\end{corollary}
\proof 
Formula (\ref{ltau1}) follows from (\ref{ltau}) and the
well known characterisation of the distribution of the first hitting
time $\tau^+_x,$ see, e.g., Bingham \cite{BIN} p. 720. 
By differentiating \eqref{ltau1}, we obtain the first part of 
\eqref{lstat2}, and using an ergodic argument--as 
in the proof of the second part of \eqref{lstat1}--we obtain
the second part of \eqref{lstat2}.
\qed

%To compute the
%mean of $L_t$ under the stationary measure we proceed as in the proof
%of Proposition \ref{loc-}. Firstly, if $X_0=x>0$ it holds that  
%$D$ is identical in law with $\tau_{-x}.$ Therefore, under the
%stationary measure we have 
We now compute the $P$-distribution of $D=\inf\{t>0:~X_t=0\}$,
by arguing as earlier: 
we have $D = \inf\{t > 0: -Y_t > X_0\}$, and, since $-Y$ is spectrally negative,
we use the hitting time formula \eqref{tau+} of Appendix B to obtain
\begin{align}
E[e^{-\theta D}] = E[ e^{-\Phi_Y(\theta) X_0}]
= \psi_Y'(0+)~\frac{\Phi_Y(\theta)}{\psi_Y(\Phi_Y(\theta))}
= \psi_Y'(0+)~\frac{\Phi_Y(\theta)}{\theta},
\label{pikk1}
\end{align}
%\begin{align}
%E[e^{-\theta D}] &= E[ e^{-\Phi_Y(\theta) X_0}]
%\nonumber
%\\
%&= \psi_Y'(0+)~\frac{\Phi_Y(\theta)}{\psi_Y(\Phi_Y(\theta))}
%\nonumber
%\\
%&= \psi_Y'(0+)~\frac{\Phi_Y(\theta)}{\theta},
%\label{pikk1}
%\end{align}
where we also used \eqref{ZOLO} and the fact that $\Phi_Y \comp \psi_Y$ is
the identity.

%we have
%  \begin{eqnarray} \label{pikk1}
%    E[e^{-\theta D}]&=& \int_0^\infty P(X_0 \in dx)\, E_0[e^{-\theta \tau_{-x}}]
%    = \int_0^\infty P(X_0 \in dx) \, e^{-\Phi_Y(\theta)x} \nonumber\\
%    &=& E[e^{-\Phi_Y(\theta) X_0}] = \psi_Y^{'}(0+)  \frac{\Phi_Y(\theta)}{\theta}.
%  \end{eqnarray} 
%Secondly, letting $\xi$ be an exponentially distributed random
%variable (with parameter $\theta$) 
%independent of $X$ we obtain usng $(\ref{pikk1})$ and $\ref{ltau1}$, 
%  \begin{eqnarray*}
%     E L_\xi &=& \int_0^\infty P(L_\xi > x) \, dx = \int_0^\infty P(L_x^{-1} \leq \xi)\, dx \\
%     &=& \int_0^\infty E [e^{-\theta L_x^{-1}}] \, dx = 
%     \int_0^\infty E[e^{-\theta D}]\, E_0 [e^{-\theta L_x^{-1}}]\, dx \\
%     &=& \psi_Y^{'}(0+) \frac{\Phi_Y(\theta)}{\theta} \int_0^\infty
%     e^{-\Phi_Y(\theta) x}\, dx = \frac{\psi_Y^{'}(0+)}{\theta},
%     \end{eqnarray*}
%and, by the uniqueness of Laplace transforms, this completes the proof.
%%$E L_t = \psi_Y^{'}(0+)\, t$, $t>0$. For $t<0$, $E L_t = - E L_{-t} = \psi_Y^{'}(0+)\, t$.
%

\section{Construction of (the stationary version of)
the fluid queue with local time input}
\label{sec3}
We wish to construct a fluid queue driven by
\[
\widehat L(s,t] = L(s,t]-(t-s),
\]
where $L$ is the local time at zero of the Markov process $X$.
The process $X$ is a stationary Markov process which
is the reflection of a spectrally negative (Section \ref{sn}) 
or a spectrally positive (Section \ref{sp}) L\'evy process.
In either case, $L$ is a stationary random measure
with rate (see (\ref{lstat1}) and (\ref{lstat2})) 
%(see Corollaries \ref{snc} and \ref{spc})
\begin{equation}
\label{ltrate}
\mu = EL(s,s+1] 
=
\begin{cases}
\Phi_Y(0), & \text{ if $Y$ is spectrally negative}\\
\psi_Y'(0+), & \text{ if $Y$ is spectrally positive}
\end{cases}.
\end{equation}
The fluid queue started from level $x$ at time $0$ is, as explained
in Appendix \ref{skororeview}, the process 
\[
(\RR_{0,t}\widehat L(x),~ t \ge 0).
\]
{From} Corollary \ref{u2} in Appendix \ref{skororeview} we immediately have:
\begin{theorem}
\label{qqq}
If $\mu < 1$ there is a unique stationary version of the fluid
queue driven by $\widehat L$ and is given by
\begin{equation}
\label{qqqeq}
Q_t = \widetilde \RR_t \widehat L = \sup_{-\infty < u \le t}
(\widehat L_t-\widehat L_u), \quad t \in \R.
\end{equation}
\end{theorem}
Thus, the following assumptions will be made throughout the paper:
\begin{description}
\item{\bf [A1]} $\mu < 1$
\item{\bf [A2]} In case $Y$ is spectrally positive and of bounded variation
then $d_Y < -1$
\end{description}
Assumption [A1] is so that we can construct a stationary version of $Q$ (as
in Theorem \ref{qqq}). If $Y$ is spectrally positive and of bounded variation
non-monotone paths then its drift $d_Y$ must be negative. If, however,
$|d_Y| \le 1$ then--see \eqref{abs-cts}--$L(s,t] \le t-s$ for all $s <t$
and so $Q$ will be identically equal to $0$.

Physically, we think of $Q$ a stationary fluid queue whose cumulative input
between times $s$ and $t$ is $L(s,t]$ and whose maximum potential
output is $t-s$.
Unlike $X$, the process $Q$ is {\em not} Markovian.
However, since $Q$ has been built on the probability space
supporting $X$, it makes sense to consider, for each $x\ge 0$,
the probability measure $P_x$ defined as $P$ conditional on $\{X_0=x\}$.

Since $L$ is a random measure allows us to consider the Palm distribution
with respect to it, namely the probability measure defined by
\[
P_L (A) = \mu^{-1} E \int_{(0,1]} \1_A \comp \theta_t L(dt).
\]
Since $P_L$ is also given as the value of the Radon-Nikod\'ym derivative
\[
P_L(A) = \mu^{-1} \frac{E \1_A L(dt)}{dt}\bigg|_{t=0}
\]
it follows that $P_L$ expresses conditioning with respect to
the event that $t=0$ is a point of increase of $L$. But $L$ is the local
time of $X$ at zero. Therefore, the support of $L$ is the set of
zeros of $X$.
We thus have
\begin{theorem}
\label{pppp}
The Palm measure $P_L$ coincides with $P_0$.
\end{theorem}
This observation allows us to use the formulae of Appendix \ref{skororeview}
involving Palm probabilities.

\section{Stationary distribution of the fluid queue}
\label{stationary-section}
We are interested in computing $P(Q_t\in \cdot)$, a probability measure
which is the same for all $t$. 
We will use three properties of $L$.
First, duality, i.e.\ that $Y$ has the same distribution when time is reversed, see (\ref{e2}),
implies that
\[
(L(0,t] ,~ t \ge 0)  \eqdist (L[t,0),~ t \le 0), \quad
\text{ under $P$ and under $P_L$}.
\]
Second, the process 
\[
L^{-1}_x = \inf\{t \ge 0: L(0,t] > x\}, \quad x \ge 0
\] 
is a subordinator under $P_0$.
Third, the Palm measure $P_L$ coincides with $P_0$.

Recall that $\psi_Y(\theta)$ has been defined as 
$\log E e^{\theta (Y_{t+1}-Y_t)}$ when $Y$ is spectrally negative
and as
$\log E e^{-\theta (Y_{t+1}-Y_t)}$ when $Y$ is spectrally positive.
The reason is that it is customary to have $\theta \ge 0$ in both cases.
Recall also that $\Phi_Y(q) = \sup\{\theta \ge 0:
\psi_Y(\theta) = q\}$.
The stationary distribution of $Q$ will be expressed in terms of $\psi_Y$.
For earlier works on this problem, we refer to \cite{Sa} for diffusion local times and \cite{Si} for the inverse of a general 
subordinator. The present formulation in Theorem \ref{distr0} is in particular tailored for the local time of $X.$ 
The proof makes use of the Palm probability which is a new ingredient. 

%h
First, since we allow discontinuous local times, the following simple lemma is needed. We omit the proof.  
%whose proof we omit:
\begin{lemma}
\label{simple}
Let $\overline \R := \R \cup \{+\infty, -\infty\}$, 
If $h: \overline \R \to \overline \R$ 
is right-continuous and non-decreasing then
$h^{-1}(x) := \inf\{t:~ h(t) > x\}$, $x \in \R$, is right-continuous
non-decreasing, $h^{-1} : \overline \R \to \overline \R$, and, for all 
$t \in \overline \R$,
\begin{equation}
\label{invcomp}
h^{-1}(h(t-)-) \le h^{-1}(h(t)-) \le t \le h^{-1}(h(t-)) \le h^{-1}(h(t)).
\end{equation}
Furthermore, $(h^{-1})^{-1} = h$.
\end{lemma}

\begin{theorem} 
\label{distr0}
(i) If $X$ is the reflection of a spectrally negative L\'evy process $Y$
with $\psi_Y'(0+) <0$, $\psi_Y(1) > 0$ then
\[
P_0(Q_0 > a) = e^{-\psi_Y(1)a},
\quad
P(Q_0 > a) = \Phi_Y(0) e^{-\psi_Y(1)a}, \quad a \ge 0.
\]
(ii) If $X$ is the reflection of a spectrally positive L\'evy process $Y$
with $0 < \psi_Y'(0+) < 1$ and $d_Y < -1$ in the case of bounded variation, then
\[
P_0(Q_0 > a) = e^{-\theta^* a},
\quad
P(Q_0 > a) = \psi_Y'(0+) e^{-\theta^* a}, \quad a \ge 0,
\]
where $\theta^* >0$ is defined by $\psi_Y(\theta^*)=\theta^*$.
\end{theorem}

\proof
By the construction of $Q$ and duality, we have
$P(Q_0 \le a)= P(\sup_{u \ge 0} (L_u -u) \le a)$. 
%At this point we need to make note of the fact that the process $\{L_t -t: t\geq 0\}$ should not have monotone paths. When, as a measure, $dL$ is singular with respect to Lebesgue measure, monotone paths are not possible. However, in the case that $Y$ is spectrally positive and of bounded variation we have seen in (\ref{abs-cts}) that local time is absolutely continuous with respect to Lebesgue measure. Moreover the paths of $\{L_t -t: t\geq 0\}$ will be monotone if and only if $d_Y\leq 1$. However this has been excluded by assumption in part (ii) of the statement of the theorem.
At this point we note that our assumptions imply that the process 
$\{L_t -t: t\geq 0\}$ does not have monotone paths.
The event $\{\sup_{t \ge 0} (L_t -t) \le a\}$ can be
expressed in terms of $L^{-1}$:
\[
\{\sup_{t \ge 0} (L_t -t) \le a\}
=
\{\sup_{x \ge 0} (x-L^{-1}_x) \le a\}.
\]
To justify this (recall that in case (i) $L$ is not necessarily 
continuous), we first assume that
$L_t \le t+a$ for all $t\ge 0$.
Hence $L_{L^{-1}_x} \le L^{-1}_x + a$ for all $x \ge 0$.
Since (Lemma \ref{simple})
\[
L_{L^{-1}_x} \ge x, \quad x \ge 0.
\] 
we have $x \le L^{-1}_x +a$ for all $x \ge 0$
and  thus $\sup_{x \ge 0} (x-L^{-1}_x) \le a$.
Assume next that $x-L_x^{-1} \le 0$ for all $x \ge 0$.
Then
$\lim_{\epsilon\downarrow 0} L^{-1}_{x-\epsilon}
= L^{-1}_{x-} \ge x-a$ for all $x \ge 0$.
Therefore, $L^{-1}_{(L_t)-} \ge L_t-a$ for all $t \ge 0$.
Since (by Lemma \ref{simple} again)
\[
L^{-1}_{(L_t)-} \le t, \quad t \ge 0,
\]
we obtain $L^{-1}_{(L_t)-} \ge t-a$ for all $t \ge 0$
and this gives $\sup_{t \ge 0} (L_t -t) \le a$.
We first compute the $P_0$-distribution of $Q_0$:
\[
P_0(Q_0 > a) = P_0(\sup_{x \ge 0} (x-L^{-1}_x) >  a),
\quad a \ge 0.
\]
Under $P_0$, the process 
\begin{equation}
\label{LAMBDA}
\{\Lambda_x := x-L^{-1}_x\,:\,  x \ge 0\},
\end{equation}
is a spectrally negative L\'evy process with bounded variation paths.
Letting
\[
\sigma_a  := \inf\{x:  \Lambda_x > a\},
\]
and applying Lemma \ref{one-sided-exit} of Appendix B, we have
\begin{equation}
\label{eeee}
P_0(Q_0 > a) = P_0(\sigma_a < \infty) 
= \lim_{q \downarrow 0} E_0\big[ e^{-q \sigma_a}  \big]
= \lim_{q \downarrow 0} e^{-\Phi_\Lambda(q) a}
= e^{-\Phi_\Lambda(0) a}.
\end{equation} 
The function $\Phi_\Lambda(q)$ is given by
\[
\Phi_\Lambda(q) = \sup\{\theta  \ge 0: \psi_\Lambda(\theta) = q\},
\quad q \ge 0
\]
(see \eqref{phi}), where 
\[
\psi_\Lambda(\theta) =\log E_0\big[ e^{\theta  \Lambda_1}  \big]
=  \theta + \log E_0\big[ e^{-\theta L^{-1}_1}  \big].
\] 
If $Y$ is spectrally negative, then we use Proposition \ref{loc-}
for an expression for
$\log E_0\big[ e^{-\theta L^{-1}_1}  \big]$. If $Y$ is spectrally positive
we use Proposition \ref{loc+}. We obtain:
\begin{equation}
\label{psil}
\psi_\Lambda(\theta) =
\begin{cases}
\theta-\frac{\theta}{\Phi_Y(\theta)}, & \text{ if $Y$ is spectrally negative},\\
\theta-\Phi_Y(\theta), & \text{ if $Y$ is spectrally positive}.
\end{cases}
\end{equation}
In both cases, $Q_0$ is exponential under $P_0$.
with parameter $\Phi_\Lambda(0)$, which has different value
in each case.
Let $\mu$ be the rate of $L$ (see (\ref{ltrate})). 
%This is also given in Lemmas \ref{loc-} and \ref{loc+}. 
Using equation \eqref{dll}, we have
\begin{equation}
\label{eeee2}
P(Q_0 > a) = \mu P_L (Q_0 > a) = \mu P_0 (Q_0 >a)
= \mu e^{-\Phi_\Lambda(0) a},
\quad a \ge 0.
\end{equation}
The proof will be complete, if we show that
\[
\Phi_\Lambda(0) =
\begin{cases}
\psi_Y(1), & \text{ if $Y$ is spectrally negative},\\
\theta^*, & \text{ if $Y$ is spectrally positive}.
\end{cases}
\]
Note that $\Phi_\Lambda(0)$ 
is the positive solution of $\psi_\Lambda(\theta)=0$.
If $Y$ is spectrally negative, we see, from the first of \eqref{psil},
that $\psi_\Lambda(\theta)=0$ iff $\Phi_Y(\theta)=1$
and, by the definition of $\Phi_Y$, the latter is true
iff $\theta=\psi_Y(1)$. Thus, $\Phi_\Lambda(0) =\psi_Y(1)$.
If $Y$ is spectrally positive,  
$\psi_\Lambda(\theta)=0$ iff $\Phi_Y(\theta)=\theta$ iff
$\theta=\psi_Y(\theta)$.
\qed

\begin{center}
\begin{minipage}{\textwidth}
{\sc Table showing the basic characteristics of the system
in both cases}\\[2mm]
\fbox{
\begin{tabular}{l||l}
%\epsfig{file=figs/spneg.eps,height=6cm}
%&
%\epsfig{file=figs/sppos.eps,height=6cm}
%\\
\begin{minipage}{6cm}
\hspace*{-5mm} \epsfig{file=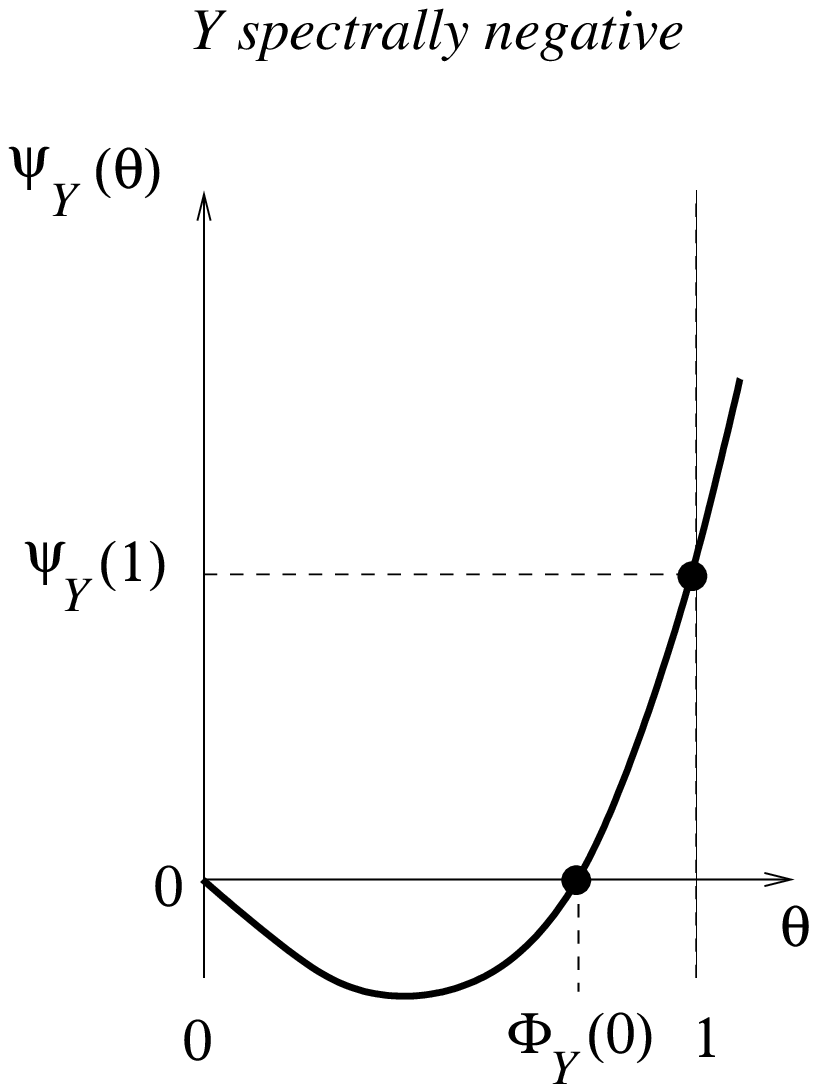,height=6cm}
\\[2mm]
$~$\hrulefill$~$
\\[2mm]
$E_0[e^{-q L_x^{-1}}] = e^{-x q/\Phi_Y(q)}$
\\[2mm]
$\mu = $ rate of $L$ = $\Phi_Y(0)$
\\[2mm]
$P_0(Q_0 > a) = e^{-\theta^* a}$
\\[2mm]
$\theta^*=\psi_Y(1)$
\\[2mm]
$P(Q_0=0)=1-\Phi_Y(0)$
\\[2mm]
$E[e^{-\theta X_0}] = \Phi_Y(0)/(\theta+\Phi_Y(0))$
\\[2mm]
$E[e^{-\theta D}] = \Phi_Y(0)/\Phi_Y(\theta)$
\end{minipage}
&
\begin{minipage}{6cm}
\epsfig{file=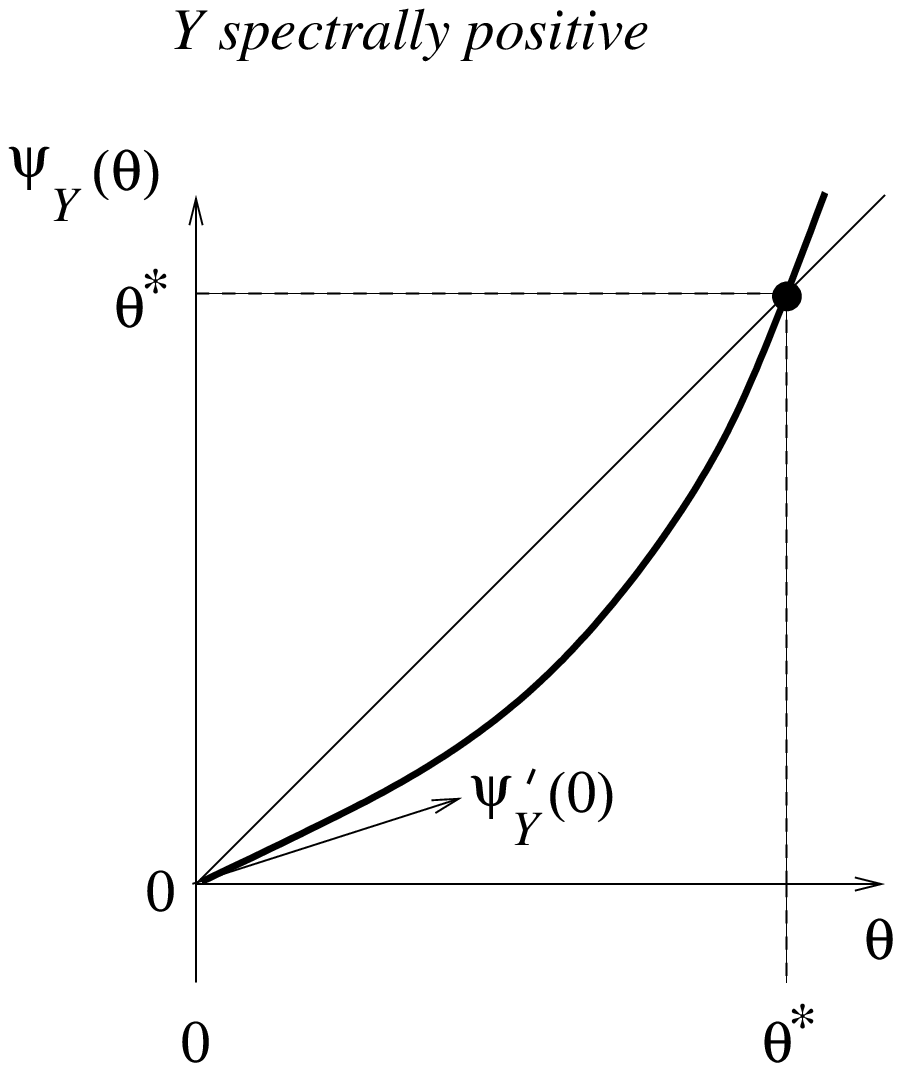,height=6cm}
\\[2mm]
$~$\hrulefill$~$
\\[2mm]
$E_0[e^{-q L_x^{-1}}] = e^{-x \Phi_Y(q)}$
\\[2mm]
$\mu = $ rate of $L$ = $\psi'_Y(0+)$
\\[2mm]
$P_0(Q_0 > a) = e^{-\theta^* a}$
\\[2mm]
$\theta^*=\psi_Y(\theta^*)$
\\[2mm]
$P(Q_0=0)=1-\psi'_Y(0)$
\\[2mm]
$E[e^{-\theta X_0}] = \psi'_Y(0+)~\theta/\psi_Y(\theta)$
\\[2mm]
$E[e^{-\theta D}] = \psi'_Y(0+)~\Phi_Y(\theta)/\theta$
\end{minipage}
\end{tabular}
}
\end{minipage}
\end{center}

\subsection{Example 1: Fluid queue driven by the local time
of a reflected Brownian motion}
Consider $Y$ to be a Brownian motion with drift (see also \cite{Sa}, \cite{MNS}, \cite{KoSa}):
\[
Y_t := \sigma B_t - \mu t, \quad t \in \R,
\]
where $\sigma > 0$, $\mu > 0$.
Here $B=(B_t, t \in \R)$ is a standard Brownian motion with
two-sided time. 
In other words, $(B_t, t \ge 0)$, $(B_{-t}, t \ge 0)$
are independent standard Brownian motions with $B_0=0$ (although
specification of $B_0$ does not affect the results below).
%Since the L\'evy measure here is $0$ (i.e.\ condition \eqref{ppp}
%does not hold) 
The L\'evy measure here is $0$.
Consider $Y$ as in Section \ref{sp} and let
\[
\psi_Y(\theta) = 
\log E [e^{-\theta Y_1}] = \frac{1}{2}\sigma^2 \theta^2 + \mu \theta,
\quad \theta > 0.
\]
Define 
\[
X_t = \widetilde \RR_t Y
= \sup_{-\infty < s \le t} 
\big( \sigma (B_t- B_s)- \mu (t-s)\big), \quad t \in \R.
\]
Lemma \ref{stationary-distn+} gives the distribution of $X_0$ under $P$:
\[
E e^{-\beta X_0} = \psi_Y'(0+) \frac{\beta}{\psi_Y(\beta)}
= \frac{\mu}{\frac{1}{2} \sigma^2 \beta + \mu},
\]
i.e.\ exponential with rate $2\mu/\sigma^2$.
Let $L$ be the local time at zero of $X$.
The rate of $L$--see \eqref{ltrate}--is $\psi_Y'(0+) = \mu$.
Assume $\mu < 1$. 
Let $\widehat L(s,t] = L(s,t]-(t-s)$ and let $Q$ be defined by
\[
Q_t = \widetilde \RR_t \widehat L,
\quad t \in \R.
\]
Theorem \ref{distr0} gives the distribution of $Q_0$ under $P$:
\[
P(Q_0>a) = \psi_Y'(0+) e^{-\theta^* a}
=
\mu e^{-2(1-\mu)a/\sigma^2}, \quad a \ge 0.
\]
Here, $\theta^*$ was found from $\psi_Y(\theta^*)=\theta^*$.
Thus $Q_0$ is a mixture of an exponential with rate $2(1-\mu)/\sigma^2$
and the constant $0$ which is assumed with probability $\mu$.

\subsection{Example 2: fluid queue driven by the local time of a compound 
Poisson process with drift}
Suppose that, for $\alpha > 0$, 
\[
Y_t = S_t-\alpha t, \quad t \in \R,
\]
where $S$ is a compound Poisson process with only positive jumps,
jump rate $\lambda$ and jump size distribution $F$. 
For simplicity, we take $F$ to be exponential
with rate $\delta >0$, i.e.\ $F(dx) = \delta e^{-\delta x} dx$. 
Then
\[
\psi_Y(\theta) = \log E e^{-\theta(Y_{t+1}-Y_t)} 
=
\alpha \theta - \lambda \int_{[0, \infty)} (1-e^{-\theta x}) F(dx)
= \alpha \theta - \frac{\lambda\theta}{\delta+\theta},
\quad \theta > 0.
\]
The assumption $0<\psi'(0+)<1$ implies that $1+\lambda/\delta>\alpha > \lambda /\delta$. Moreover, the assumption $|d_Y|>1$ implies additionally that $\alpha>1$. %where $m = \int x F(dx)$,
We can define the
background stationary Markov process by
\[
X_t = \widetilde R_t Y 
=
\sup_{-\infty < s \le t} (S_t-S_s-\alpha(t-s)), \quad t \in \R.
\]
We have
\[
E e ^{-\beta X_0} = 
\frac{\alpha - \lambda m}{\alpha - \lambda \int_{[0, \infty)} 
\frac{1-e^{-\beta x}}{\beta} F(dx)}.
\]
Unlike the previous example, here $P(X_0 =0) = \lim_{\beta\uparrow \infty}
E e ^{-\beta X_0} = \alpha-\lambda /\delta > 0$.
The local time $L$ of $X$ at $0$ has rate
\[
\mu = \psi_Y'(0+) = \alpha-\lambda /\delta.
\]
The assumptions on $\alpha$ imply that $\mu < 1$ and hence we can construct the stationary process $Q$
by $Q_t = \widetilde \RR_t \widehat L$, where $\widehat L(s,t] = L(s,t]
-(t-s)$.
We have
\[
P(Q_0 > x) = \mu e^{-\theta^* x}
\]
where $\theta^* = \psi_Y(\theta^*)= \lambda(\alpha-1)^{-1} - \delta$. Note that the latter is positive since $\alpha>1$ and $1+\lambda/\delta >\alpha$.

\subsection{Example 3: 
fluid queue driven by the local time of a risk-type process}
Let
\[
Y_t = bt - S_t, \quad t \in \R,
\]
where $b>0$, $S$ is an $\alpha$-stable subordinator, $0 < \alpha < 1$,
with
\[
E e^{-\theta (S_{t+1}-S_t)} = e^{-c \theta^\alpha}, \quad \theta > 0,
\]
and $c$ is a positive constant. Thus, $S$ is 
a $(1/\alpha)$-self-similar process,
i.e.\ $(S_{\kappa t}, t \in \R) \eqdist (\kappa^{1/\alpha} S_t, t \in \R)$.
We here have $E S_t = +\infty$ for $t > 0$ and $S_t \to \infty$ faster
than linearly, so $Y_t \to -\infty$, as $t \to \infty$, a.s.
Similarly, $S_t \to -\infty$ as $t \to -\infty$, a.s. 
So the stationary reflection of $Y$
\[
X_t = \widetilde \RR_t Y
=
\sup_{-\infty < s \le t} \big(b(t-s) - (S_t-S_s) \big),
\quad t \in \R,
\]
exists uniquely, due to Lemma \ref{u1}.
Physically, $X_t$ is the content of a queue with linear input (arriving
at rate $b$) and jump-type service represented by $S$.
Alternatively, $X$ is a so-called risk process in the theory of risk.
We have
\[
\psi_Y(\theta) = 
\log E e^{-\theta(Y_{t+1}-Y_t)}
= b \theta - c \theta^\alpha, \quad \theta \ge 0.
\]
We refer to Section \ref{sn} and, specifically, Lemma \ref{stationary-distn},
for the distribution of $X_0$ which is exponential with rate 
$\mu > 0$ where $\mu$ satisfies $\psi_Y(\mu) =0$, i.e.\ 
\[
\mu = (c/b)^{1/(1-\alpha)}.
\]
The local time $L$ of $X$ at zero is such that $t \mapsto L(0,t]$ is
a.s.\ right-continuous (but not continuous)
with rate $\mu$. Assuming that $\mu < 1$, or
\[
c < b,
\]
we can further let $\widehat L(s,t] = L(s,t]-(t-s)$ and let $Q$ be defined by
\[
Q_t = \widetilde \RR_t \widehat L, \quad t \in \R
\]
(see Lemma \ref{qqq}.)
Theorem \ref{distr0} gives the distribution of $Q_0$ under $P_0$ and under $P$.
We have,
\[
P_0(Q_0 > x) = e^{-(b-c) x},
\quad
P(Q_0 > x) = \left(\frac{c}{b}\right)^{\frac{1}{1-\alpha}}
e^{-(b-c) x}, \quad x \ge 0.
\]

\subsection{Example 4: fluid queue driven by the local time
of a risk-type process with a Brownian component}
Take
\[
Y_t = 3bt +\sigma B_t - S_t,
\]
where $S$ is the inverse local time of an independent Brownian motion. 
Assume $\sigma^2 > 0$. We have
\[
\psi_Y(\theta) = 
\log E e^{\theta(Y_1-Y_0)}
= 3b \theta + \frac{1}{2} \sigma^2 \theta^2 - 2c \theta^{1/2},
\quad \theta >0,
\]
where $c$ is a scaling parameter.
Since $\lim_{t \to \infty} Y_t = - \infty$, a.s.,
Lemma \ref{u1} allows us to construct
$X_t = \widetilde \RR_t Y$. 
Here, $Y$ is spectrally negative, and so, 
as shown in Lemma \ref{stationary-distn}
\[
P(X_0 > x) = e^{-\mu x}, \quad x >0,
\]
where $\mu >0$ and
\[
\psi_Y(\mu)=0.
\]
Letting 
\[
\delta = 1+ b^3 \sigma^{-2} c^{-2},
\]
we find
\[
\mu = 2 \bigg(\frac{c}{\sigma^2}\bigg)^{2/3}
\frac{(\delta^7+1)^{1/3} + (\delta^7-1)^{1/3}}{\delta^2}.
\]
Here, $P(X_0=0)=0$.
As in Proposition \ref{loc-},
this $\mu$ is the rate of the local time $L$ of $X$.
Note that, since $Y$ has unbounded variation paths,
the local time $L$ is a.s.\ continuous.
Assuming that $\mu <1$, which is equivalent to 
\[
\psi_Y(1) = b + \frac{1}{2} \sigma^2-c > 0,
\]
we construct $Q$ as before: 
$Q_t = \widetilde \RR_t \widehat L, t \in \R$. 
{From} Theorem \ref{distr0} we have
that
\[
P(Q_0 > x) = \mu e^{-(b + \frac{1}{2} \sigma^2-c) x}, \quad x > 0.
\]

%\subsection{A general principle}
%
%The last two examples show that the stationary distribution is always
%exponentially distributed with an atom at zero. This is no coincidence
%and in fact will always appear in the following set up.
%
%\begin{theorem} 
%Suppose that $M$ is any Markov process on state space $\mathcal{M}$
%with expectation operators $\{\mathbb{E}_x : x\in \mathcal{M}\}$.
%Denote by $L$ the Markov local time at a particular state
%$\partial\in\mathcal{M}$. Suppose further that $\partial$ is a
%recurrent state and that $M$ has a stationary distribution on
%$\mathcal{M}$ which we denote by $\mu$.  If $\mathbf{P}$ is the law
%of the stationary distribution of $S$ when driven by $L$. Then if
%$\mathbf{P}$ we have that for all $t\geq 0$ and $a>0$
%\[
%\mathbf{P}(S_t>a)
%= e^{-\Phi(0)a}\int_{\mathcal{M}}\mathbb{E}_x(e^{- \Phi(0) \tau^\partial})\mu(dx)
%\]
%where $\tau^\partial = \inf\{t>0 : M_t = \partial\}$ and 
%\[
%\Phi(0)= \sup\{\theta>0 : \psi(\theta)=0\}
%\]
%with $\psi$ being the Laplace exponent of the L\'evy process $\{t -
%L^{-1}_t : t\geq 0\}$. Hence $\mathbf{P}$ exists if and only if
%$\Phi(0)>0$.
%\end{theorem}
%\begin{proof}
%There is little deviation from the programme above. Let $Z$ be the
%L\'evy process with Laplace exponent $\psi$. Then
%\[
%\mathbb{P}_\partial(S_0 >a) = \mathbb{P}(\tau^+_a(Z)<\infty)   e^{-\Phi(0)a}.
%\]
%Assuming that $\Phi(0)>0$ then for $x\in\mathcal{M}$
%\[
%\mathbb{P}(S_0 >a| M =x) = e^{-\Phi(0)a}\mathbb{E}_x(e^{-\Phi(0)\tau^\partial}).
%\]
%Hence integrating against the stationary distribution finished the
%proof.
%\end{proof}

\section{Idle and busy periods}
\label{sec5}
In this section we study idle and busy periods of the fluid queue process
$\{Q_t\,:\, t\in\R\}$ as defined in  (\ref{qqqeq}). 
We work under the assumptions that $Y$ is either spectrally
negative or spectrally positive and that $0 < \mu < 1$, where
$\mu$ is the rate of $L$--see \eqref{ltrate}.
Under these assumptions, the process $Q$ constructed above is
stationary and the sets
\begin{align*}
\{t \in \R:~ Q_{t-} >0, ~ Q_t=0\} \equiv \{g(n):~ n \in \Z\}\\
\{t \in \R:~ Q_{t-} =0, ~ Q_t>0\} \equiv \{d(n):~ n \in \Z\}
\end{align*}
are a.s.\ discrete, with elements denoted $g(n)$, $d(n)$, respectively.
We need a convention for their enumeration, and here is the one we adopt.
%i.e.\ $g(1) = \min\{t>0:~ Q_{t-}>0,~ Q_t =0\}$,
%$g(2) = \min\{t > g(1):~ Q_{t-}>0,~ Q_t =0\}$, etc.

%%%%%%%%%%% BEGIN FIGURES %%%%%%%%%%%%%%%%%%%%%%%%%%%%%%
\begin{figure}[h]
\begin{center}
\epsfig{file=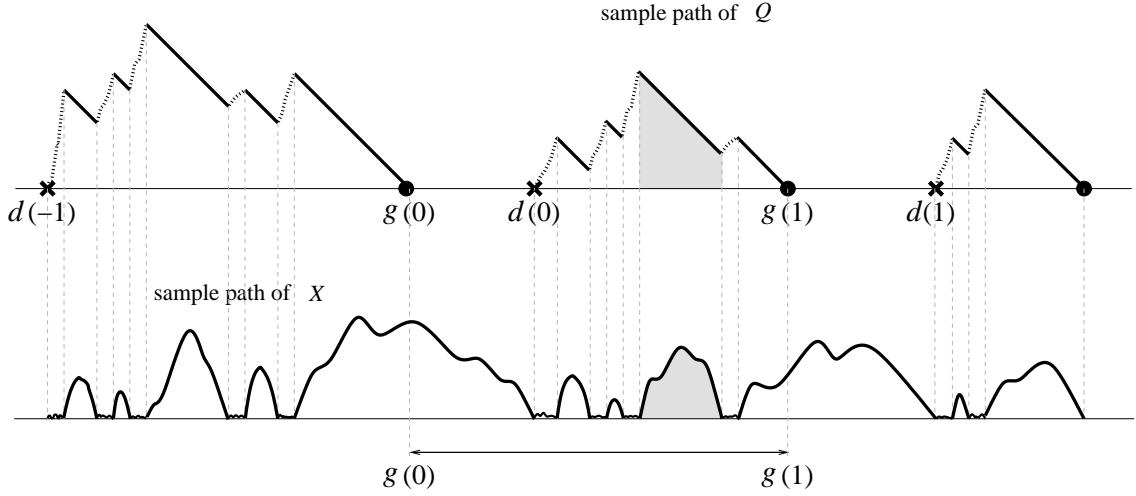,width=15cm}
\end{center}
\caption{\em Typical behaviour of $Q$ and the
background Markov process $X$ when the underlying L\'evy process $Y$
has unbounded variation paths. By convention, the origin of
time is contained between $g(0)$ and $d(0)$. Note that
excursions of $X$ away from $0$ correspond to intervals over which $Q$
decreases.}
%
%conditional on $\{Q_0=0\}$.
%The observed idle period is the interval $(g,d)$ straddling $0$.
%The excursion of $X$ containing $0$ is the piece
%$(X_t, G \le t \le D)$ of the sample path of $X$.
%(The random variables $Q_G, G, D$ are independent conditional on $X_0$.)}
%{\color{red} This is an ATTEMPT for a meaningful figure. Andreas says
%it is `wrong' and needs artistic modification. Why?}
\end{figure}

\begin{figure}[h]
\begin{center}
\epsfig{file=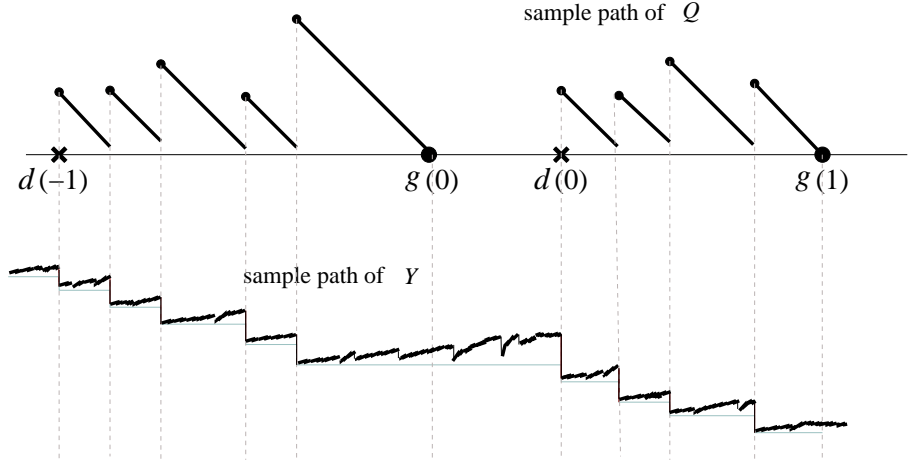,width=12cm}
\end{center}
\caption{\em Typical behaviour of $Q$ and the
background L\'evy process $Y$, in case that $Y$ is spectrally
negative with bounded variation paths.}
\end{figure}

\begin{figure}[h]
\begin{center}
\epsfig{file=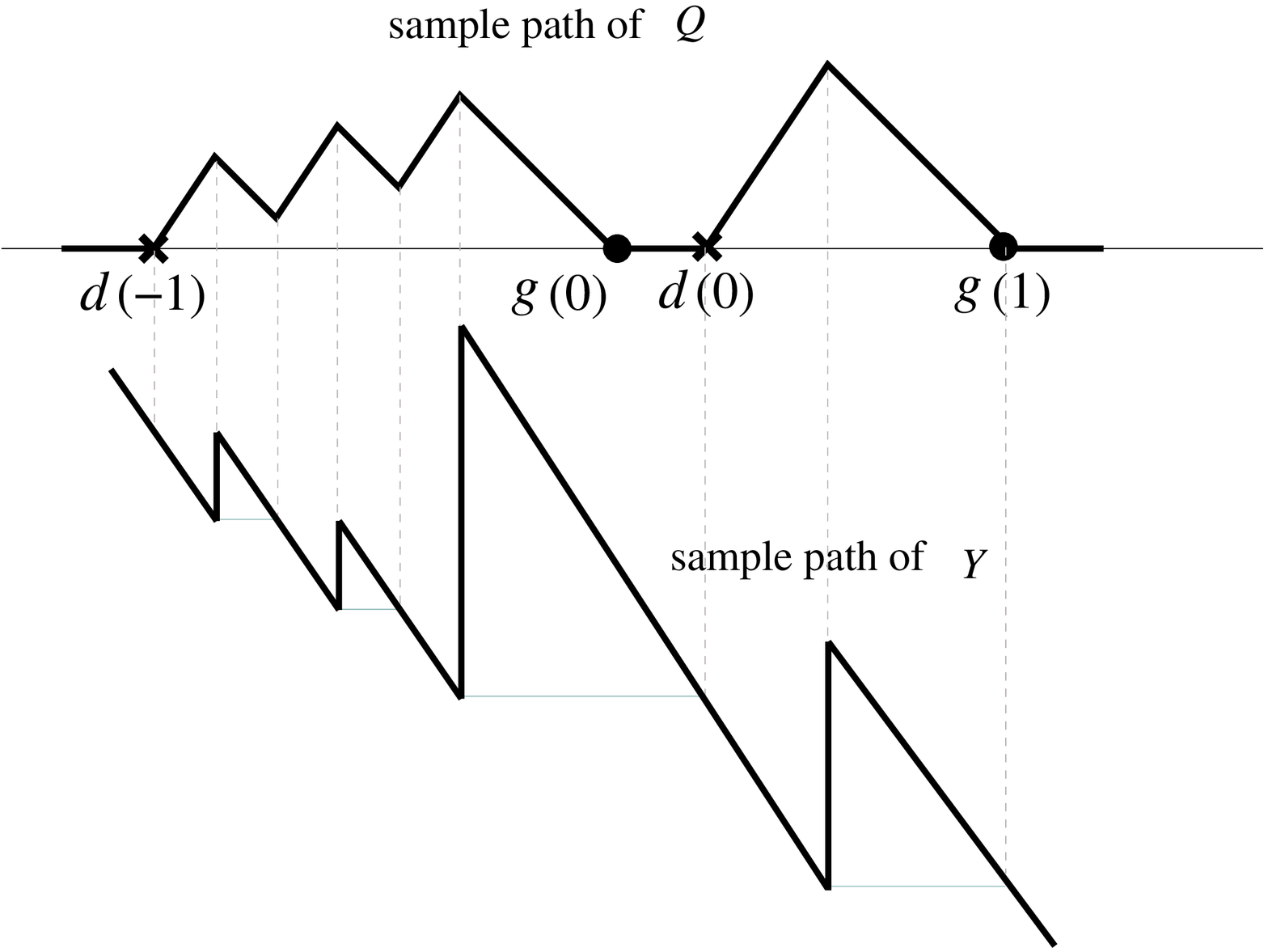,width=8cm}
\end{center}
\caption{\em Typical behaviour of $Q$ and the
background L\'evy process $Y$, in case that $Y$ is spectrally
positive with bounded variation paths. Here, only the case where the
jump part of $Y$ is compound Poisson is depicted. When $Q_t>0$ and $X_t=0$,
we see that $Q_t$ increases at rate $|d_Y|-1$, where $d_Y<-1$ is the
drift of $Y$.}
\end{figure}
%%%%%%%%%%% END FIGURES %%%%%%%%%%%%%%%%%%%%%%%%%%%%%%

First,
\[
\cdots < g(-1) < g(0) \le 0 < g(1) < g(2) < \cdots,
\]
Second,
\[
g(n) < d(n) < g(n+1), \quad n \in \Z.
\]
%An {\em idle} ({\em busy}) period of $Q$ is a maximal interval of 
%time $I$ such 
%that $Q_t=0$ ($Q_t>0$) for all $t$ in the interior of $I$.
%An {\em observed idle (busy) period} at time $t_0$ is a random
%variable distributed as an idle (busy) period $I$ 
%conditional on $\{t_0 \in I\}$. 
Let $N_1$ (resp.\ $N_2$) be the random measure (point process) that
puts mass $1$ to each of the points $g(n)$ (resp.\ $d(n)$).
Notice that the point processes $N_1, N_2$ 
are jointly stationary. 

The intervals $(g(n), d(n))$ are called {\em idle periods},
while the intervals $(d(n), g(n+1))$ are called 
{\em busy periods}.
An {\em observed idle period} is, by definition, equal in distribution
to an idle period, given that the idle period contains a fixed time
$t$ of observation. By stationarity, we may take the time of observation
to be $t=0$. In other words,
\[
\text{observed idle period} := \big(~ (g(0), d(0)) \mid Q_0=0 \big)
\eqdist \big(~ (g(0), d(0)) \mid g(0) < 0 < d(0) \big).
\]
Here, $\eqdist$ denotes equality in distribution under measure $P$. 
Similarly,
\begin{equation}
\label{oBP}
\text{observed busy period} := \big(~ (d(0), g(1)) \mid Q_0>0 \big)
\eqdist \big(~ (d(0), g(1)) \mid d(0) < 0 < g(1) \big).
\end{equation}

In this section we identify the distribution of observed idle (and busy) 
periods (see Propositions \ref{pikk8}, \ref{pikk12}). 
%
%In the main results of this section we calculate the joint 
%Laplace transforms of $(g(0),d(0))$ conditional on $Q_0=0$ 
%(see Proposition  \ref{pikk8}) and
%of $(d(0),g(1))$ conditional on $Q_0 >0$ (see Proposition  \ref{pikk12}). 
%hh
In both cases we shall appeal to the result of Lemma \ref{pikk2} below
a short proof of which is provided. 
Note that formula \eqref{b1} below and related facts are also
proved in Kozlova and Salminen \cite[Section 5]{KoSa} 
and Sirvi\"o (Kozlova) \cite{Si}.

\begin{lemma}
\label{pikk2}
Let $(\Omega, \FF, P)$ be a probability space endowed with a $P$-preserving
flow $(\theta_t, t \in \R)$ [see Appendix A].
Let $N_1, N_2$ be jointly stationary simple random point processes 
($N_i\comp \theta_t (B) = N_i(B+t)$, $t \in \R$, $B \in \BB(\R)$, $i=1,2$)
with points $\{t_i(n), n \in \Z\}$, $i=1,2$, such that
\[
\cdots < t_1(-1) < t_1(0) \le 0 < t_1(1) < t_1(2) < \cdots,
\]
and 
\[
t_1(n) < t_2(n) < t_1(n+1), \quad \text{ for all $n \in \Z$.}
\]
Let $M$ be the random measure which is defined through its derivative with 
respect to the Lebesgue measure as 
\[
M(dt)/dt = \sum_{n \in \Z}
\1\big(t_1(n) < t < t_2(n)\big).
\]
Assume that $N_1$ has finite intensity. Let $P_M$ be the Palm measure with
respect to $M$. Then
\begin{equation}
\label{b1}
E_M[e^{-\alpha t_2(0) + \beta t_1(0)}]
= \frac{\alpha E_M[e^{-\alpha t_2(0)}] - \beta E_M[e^{-\beta t_2(0)}]  }
{\alpha-\beta},
\quad \alpha, \beta > 0.
\end{equation}
\end{lemma}
\proof
It is easy to see that $M$ is also stationary, i.e.\ $M\comp \theta_t (B)
= M(B+t)$. 
Let $P_{N_i}$ be the Palm measure with respect to $N_i$, $i=1,2$
and let $\lambda$ be the intensity of $N_1$ (which is--due
to the law of large numbers--the same as the intensity of $N_2$).
It follows easily from Campbell's formula that
$M$ has finite intensity: 
$\lambda_M = \lambda E_{N_1} [t_2(0)-t_1(0)] < \infty$.
The Palm exchange formula\footnote{In \cite{BB} 
p. 21 the formula is given and proved for point processes 
but the generalization for 
arbitrary jointly stationary random measures is straightforward.}
 between $P_M$ and $P_{N_1}$ yields
\begin{equation}
\label{exch}
\lambda_M E_M [Y] = 
\lambda E_{N_1} \int_{t_1(0)}^{t_1(1)} Y \comp \theta_t ~ M(dt)
=
\lambda E_{N_1} \int_{t_1(0)}^{t_2(0)} Y \comp \theta_t ~ dt,
\end{equation}
for any bounded random variable $Y$.
Apply \eqref{exch} with $Y=e^{-\alpha t_2(0) + \beta t_1(0)}$. 
Since $t_1(0) = \sup\{t \le 0: N_1(\{t\})=1\}$,
$t_2(0) = \inf\{t > t_1(0): N_2(\{t\})=1\}$,
and $P_{N_1}(t_1(0)=0<t_2(0))=1$,
we have
$t_1(0)\comp \theta_t = -t$, $t_2(0)\comp \theta_t = t_2(0)-t$,
$P_{N_1}$-a.s.\ on $\{t_1(0) < t < t_2(0)\}$.
Therefore, $Y\comp\theta_t = e^{-\alpha t_1(0)} e^{(\alpha-\beta)t}$,
$P_{N_1}$-a.s.\ on $\{t_1(0) < t < t_2(0)\}$, and so
\begin{equation}
\label{MN1}
\lambda_M E_M [e^{-\alpha t_2(0) + \beta t_1(0)}]
= \lambda \frac{E_{N_1}[e^{-\beta t_2(0)}]-E_{N_1}[e^{-\alpha t_2(0)}]}
{\alpha-\beta}.
\end{equation}
Arguing in a similar manner, through the exchange formula between
$M$ and $N_2$, we obtain
\begin{equation}
\label{MN2}
\lambda_M E_M [e^{-\beta t_2(0) + \alpha t_1(0)}]
= \lambda \frac{E_{N_2}[e^{\alpha t_1(0)}]-E_{N_2}[e^{\beta t_1(0)}]}
{\beta-\alpha}.
\end{equation}
%Comparing the two, we find
%\[
%E_{N_1}[e^{-\beta t_2(0)}] = E_{N_2}[e^{\beta t_1(0)}],
%\quad
%E_{N_1}[e^{-\alpha t_2(0)}] = E_{N_2}[e^{\beta t_1(0)}].
%\]
Setting $\beta=0$, and then $\alpha=0$, in \eqref{MN1} we obtain:
\begin{align}
\lambda_M E_M[e^{-\alpha t_2(0)}] &= \lambda
\frac{1-E_{N_1}[e^{-\alpha t_2(0)}]}{\alpha}],
\label{aaa}
\\
\lambda_M E_M[e^{\beta t_1(0)}] &= \lambda
\frac{1 - E_{N_1}[e^{-\beta t_2(0)}]}{\beta}.
\label{bbb}
\end{align}
On the other hand, with $\alpha=0$ in \eqref{MN2}, we have
\begin{equation}
\label{bbb2}
\lambda_M E_M[e^{-\beta t_2(0)}] 
= \lambda
\frac{1-E_{N_2}[e^{\beta t_1(0)}]}{\beta}.
\end{equation}
The exchange formula between $N_1$ and $N_2$ shows that
the right hand sides of \eqref{bbb} and \eqref{bbb2} are equal.
Substituting these and \eqref{aaa} into \eqref{MN1}
we obtain the result.
\qed

%Taking $t_0=0$ as the `observation point' let
%\begin{align*}
%g_i &:= \sup \{t < 0: Q_t > 0\},\\
%d_i &:= \inf \{t > 0: Q_t > 0\}.
%\end{align*}
%Then
%\[
%\text{observed idle period} :\eqdist \big( (g_i, d_i) \mid Q_0 =0 \big).
%\]
%We define similarly
%\begin{align*}
%g_b &:= \sup \{t < 0: Q_t = 0\},\\
%d_b &:= \inf \{t > 0: Q_t = 0\},
%\end{align*}
%and
%\[
%\text{observed busy period} :\eqdist \big( (g_b, d_b) \mid Q_0 >0 \big).
%\]
%
%In the main results of this section we calculate the joint Laplace transforms of $(g_i,d_i)$ (see Proposition  \ref{pikk8}) and
% $(g_b,d_b)$ (see Proposition  \ref{pikk12}) under the appropriate conditioning specified above. 
%%hh
%In both cases we shall appeal to the following result, which holds for excursions of general stationary processes, see Pitman \cite{PIT} Corollary p. 298 
%and references therein. For formula (\ref{b1}) below and related facts we refer to 
%Kozlova and Salminen \cite{KoSa} Section 5 and Sirvi\"o (Kozlova) \cite{Si}.
%\begin{lemma} \label{pikk2}
%Let the pair $(g,d)$ be either $(g_i,d_i)$ or $(g_b, d_b)$. Then
%$$
%-g \eqdist d,
%$$
%and given
%$v:= d-g$,  the random variable $-g$ is  uniformly distributed on $(0,v)$.
%  In particular,
%\begin{equation}
%\label{b1}
%E [e^{-\alpha d + \beta g}] 
%= \frac{\alpha E[e^{-\alpha d}] - \beta E[e^{-\beta d}]}{\alpha -\beta},
%\quad \alpha, \beta \ge 0.
%%= \frac{\alpha E[e^{\alpha G}] - \beta E[e^{\beta G}]}{\alpha -\beta}.$$
%\end{equation}
%\end{lemma}
%
%{\bf The formulation here is not quite right - one needs to condition on something like $-g,d>0$ ?}
%

\subsection{Observed idle periods}
\label{ssec51}

We are interested in the distribution of the idle period $(g(0), d(0))$,
given that $Q_0=0$. 
The rationale used for this computation is as follows:
We can always assume that $X_0>0$, since this is an event
with probability $1$. If $Q_0=0$ then $Q_t$ will remain $0$ at least until
$X$ hits $0$, since $Q$ cannot increase unless there is
an accumulation of local time $L$, and this can happen only
when $X$ is $0$. Recall that the first hitting time of $0$ by $X$
is denoted by $D=\inf\{t>0: X_t=0\}$.
If $Q_0=0$, the first time that $Q$ becomes positive has been
denoted by $d(0)$.
Our claim is:
\begin{lemma} 
\label{iniin}
Given that $Q_0=0$ the ending time of the idle period is a.s. equal to $D,$ i.e., 
\label{Dd}
\[
P(d(0) = D \mid Q_0=0) =1.
\]
\end{lemma}
\proof
{From} the argument above we have that $D \le d(0)$ a.s.\ on $\{Q_0=0\}$.
Suppose that there is $\Omega_0 \subset \Omega$ with $P(\Omega_0)>0$
such that $Q_0=0$ and $D < d(0)$ a.s.\ on $\Omega_0$.
If $Q_0=0$ and $D < d(0)$ then
\[
Q_t = \sup_{D \le u \le t} \{L(u,t]-(t-u)\} \equiv 0,
\quad
\text{ for all } t \in (D, d(0)).
\]
This implies that
\[
L(u,t] \le t-u, \quad
\text{ for all } D < u < t < d(0),
\]
which means that, for $\omega \in \Omega_0$,
$L(\omega, \cdot)$ is absolutely continuous on some right neighbourhood
of $D$. If $Y$ is spectrally negative or if $Y$ is spectrally
positive but not of bounded variation, then $L$ is a.s.\ singular
on any  right neighbourhood of $D$, and we obtain a contradiction.
If $Y$ is spectrally positive with bounded variation paths
then $L$ is absolutely continuous and is given by \eqref{abs-cts}.
In this case, $Q$ increases at rate $|d_Y|-1>0$ 
%{\bf (I THINK THIS IS AN ASSUMPTION WE MUST MAKE, OTHERWISE
%$Q\equiv 0$ AT ALL TIMES)}
whenever it is positive
and this shows immediately that here, too, $D=d(0)$ a.s.\ on $\{Q_0=0\}$.
\qed 
\begin{remark} {\rm
The result in Lemma \ref{iniin} can alternatively be expressed by saying that the process $(L_t-t\,:\, t\geq 0)$ is 
under $P_0$ initially increasing. In \cite{MNS} Proposition 6.3 this is proved in the case $X$ is a reflecting Brownian
motion, and the proof therein could have been modified to cover the present case. However, we found it 
motivated to give the above proof which highlights other aspects than the proof in \cite{MNS}.}  
\end{remark}

Using Lemma \ref{pikk2},
we shall reduce the problem to that of finding the distribution of 
$D=\inf\{t>0: X_t=0\}$ given that $Q_0=0$.
Let $N_1$ (resp.\ $N_2$) be the point process with points $\{g(n\}$ 
(resp.\ $\{d(n)\}$).
Then $M(dt)/dt = \1(Q_t =0)$, and so
\[
P_M = P(\cdot \mid Q_0=0).
\]
Formula \eqref{b1}, together with Lemma \ref{Dd}, then gives
\begin{equation}
\label{almostthere}
E[e^{-\alpha d(0)+\beta g(0)} \mid Q_0=0] =
\frac{\alpha E[e^{-\alpha D}|Q_0=0] - \beta E[e^{-\beta D}|Q_0=0] }{\alpha-\beta}.
\end{equation}
To compute the distribution of $D$ given $Q_0=0$ we need the following two lemmas.
\begin{lemma}
\label{notice} 
Let \[
G := \sup\{t <0:~ X_t=0\}.
\]
Then it holds
\[
\{Q_0=0\} = \{Q_G + G \le 0\}.
\]
%where
%\[
%G := \sup\{t <0:~ X_t=0\}.
%\]
\end{lemma}
\proof
%The claimed identity is intuitively clear since the local time $L$ does not increase in the interval $(G,D).$  
Since $X_t >0$ for all $t \in (G,D)$,
we have
\[
L(s,t] =0, \quad G \le s \le t \le D.
\]
Recall that
\[
Q_t = \RR_{s,t} \widehat L (Q_s)
=
\sup_{s \le u \le t} \widehat L(u,t] \vee \big(Q_s + \widehat L(s,t]\big)
\]
So, if $G \le s \le t \le D$, we have $\widehat L(s,t] =
L(s,t]-(t-s) = -(t-s)$, i.e.
\[
Q_t = (Q_s - (t-s))^+, \quad G \le s \le t \le D.
\]
If we assume that $Q_0=0$, we have $G \le g$ and so
\[
0=Q_g = (Q_G-(g-G))^+,
\]
which implies that
$Q_G +G = g \le 0$.
\qed

%We have
%
%{\color{red} ---------- SUPPLY PROOF ----------}
%
%
%Recall that we are interested in 
%\[
%\text{observed idle period} :\eqdist \big( (g_i, d_i) \mid Q_0 =0 \big).
%\]
%Since $P(Q_0=0)=1-\mu >0$, conditioning on $Q_0=0$ is simply a na\"ive conditioning.
%%We keep in mind throughout  that $X, Q$ are jointly stationary
%%processes.
%{From} the construction of $Q$--see \eqref{qqqeq}, \eqref{sds},
%\eqref{sds2}--it is clear that $Q$ increases only on the support of $L$
%which is the closure of the set $\{t: X_t=0\}$.
%Hence, for all $t$ in the interior of an idle period, $X_t > 0$.
%Define then 
%\begin{align*}
%G &:= \sup \{t < 0: X_t = 0\},\\
%D &:= \inf \{t > 0: X_t = 0\}.
%\end{align*}
%At this point it is worth keeping in mind that $X_0$ may have an
%atom at zero (as, e.g.\ in Example 2).
%Since $X_t =0$ implies that $t$ is in the support of $L$
%which implies that $t$ is a point of increase of $Q$, it
%follows that 
%\[
%d_i=D.
%\]
%
\begin{lemma}
\label{3indep}
(i) Conditional on $X_0$, the random variables $Q_G, G, D$ are
independent (under $P$).
\\
(ii)
For all $x \ge 0$, $t \ge 0$, \[
P(Q_G > t) = P_x(Q_G > t) = P_0(Q_0 >t) = e^{-\theta^* t}. \]
where $\theta^* = \psi_Y(1)$ if $Y$ is spectrally negative
or is equal to the unique positive solution of $\theta^*=\psi_Y(\theta^*)$
if $Y$ is spectrally positive.
\\
(iii)
$Q_G$ is independent of $(G,D)$ (under $P$).
\end{lemma}
\proof
(i) The independence follows from the strong Markov
property at $G$ (at which $X_G=0$) and Markov property
at $0$. Indeed, first observe that $G$ is
a stopping time with respect to the
filtration $\{\FF_t :=\sigma(X_{-s} , 0 \le s \le t),~ t \ge 0\}$.
Second, $Q_G = \widetilde \RR_G \widehat L
= \sup_{s \le G} \widehat L(s, G]
= \sup_{s \le G} \big( L(s, G] - (G-s) \big)$
and so $Q_G \1(G < t)$ is measurable with respect to
$\FF'_t :=\sigma(X_{-s} , s > t)$ for all $t$.
This proves independence between $Q_G$ and $G$.
Third, $D$ is measurable with respect to
$\FF''_0 = \sigma(X_s, s \ge 0)$.
So, conditionally on $X_0$, the random variable $D$
is independent of the pair $(Q_G, G)$.
(ii) The distribution statement about $Q_G$ follows from the strong
Markov property at $G$. Let, as usual, $\FF_{-G} =
\{A \in \sigma(X_{-s}, s \ge 0): ~ A \cap \{-G \le t\} \in \FF_{-t}\}$.
Since $Q_G = Q_0 \comp \theta_G$,
\[
P(Q_G > t) = P(Q_0 \comp \theta_G > t)
= E P(Q_0 \comp \theta_G > t \mid \FF_{-G})
= E P_{X_G}(Q_0 > t) = P_0(Q_0 > t) = e^{-\theta^* t}
\]
where the latter follows from Theorem \ref{distr0}.
(iii) This is immediate from (i) and (ii).
\qed

%Referring back to the formulae (\ref{pikk6}) and (\ref{pikk1}) the quantity $E[e^{-\alpha D}]$ has already been computed in both the spectrally negative and positive cases.
%Lemmas \ref{pikk2}--\ref{3indep} are combined and lead the following result.
\begin{proposition}[distribution of observed idle period] 
\label{pikk8}
Fix $\alpha, \beta \ge 0, \quad \alpha \ne \beta$.
\begin{description}
\item[(i)] When $Y$ is spectrally negative we have 
\begin{multline*}
  E \big[ e^{-\alpha d(0) + \beta g(0)} \mid Q_0=0\big]  
\\
= \frac{\Phi_Y(0)}{1-\Phi_Y(0)}~
  \frac{\psi_Y(1)}{\alpha -\beta}~ \left(
      \frac{\alpha}{\alpha - \psi_Y(1) }\frac{\Phi_Y(\alpha)-1}{ \Phi_Y(\alpha)}
      - \frac{\beta}{\beta - \psi_Y(1) }\frac{\Phi_Y(\beta) -1}{ \Phi_Y(\beta)}\right).
\end{multline*}
\item[(ii)] When $Y$ is spectrally positive we have
\begin{equation*}
  E \big[ e^{-\alpha d(0) + \beta g(0)} \mid Q_0=0\big] 
= 
\frac{\psi_Y'(0+)}{1-\psi_Y'(0+)}~
  \frac{\theta^*}{\alpha -\beta}~ \left(
      \frac{\alpha - \Phi_Y(\alpha)}{\alpha -\theta^* }
      - \frac{\beta-\Phi_Y(\beta)}{\beta - \theta^* }
      \right),
\end{equation*}
 where $\theta^*>0$ is defined by $\psi_Y(\theta^*) =\theta^*$.
 \end{description}
 \end{proposition}     
\proof
{From} Lemma \ref{3indep} we have
that $D, Q_G$ are conditionally independent given $X_0$ and $G$. Hence
\begin{align*}
E[e^{-\theta D} \1(Q_G+G \le 0) \mid X_0, G]
&= E[e^{-\theta D} \mid X_0, G]
~ P(Q_G \le -G \mid X_0, G).
\\
&= E[e^{-\theta D} \mid X_0, G]
~ (1-e^{\theta^* G})
\\
&= E[ e^{-\theta D} - e^{-\theta D + \theta^* G} \mid X_0, G],
\end{align*}
where the second line was obtained from the facts (all consequences
of Lemma \ref{3indep}) that 
(i) $D, G$ are conditionally independent given $X_0$, 
(ii) $Q_G, G$ are also conditionally independent given $X_0$,
and (iii) $Q_G$ is independent of $X_0$ and exponentially
distributed with parameter 
\begin{equation}
\label{th*}
\theta^*
= 
\begin{cases}
\psi_Y(1), & \text{ if $Y$ is spectrally negative}\\
\psi_Y(\theta^*), & \text{ if $Y$ is spectrally positive}
\end{cases}
\end{equation}
Taking expectations we get
\[
E[e^{-\theta D} \1(Q_G+G \le 0)] 
= E[ e^{-\theta D} - e^{-\theta D + \theta^* G}],
\]
and, using Lemma \ref{notice},
\[
E[e^{-\theta D} \mid Q_0=0] =
\frac{E[ e^{-\theta D}] - E[e^{-\theta D + \theta^* G}]}{P(Q_0=0)}.
\]
We now use a version of Lemma \ref{pikk2}, formulated
for excursions of general stationary processes;
Pitman \cite[Corollary p.\ 298; references therein]{PIT}. The random
measures $N_1, N_2$ correspond to the beginnings and ends of excursions
of the stationary process $(X_t, t \in \R)$, and, since there is never
an interval of time over which $X$ is zero, the random measure $M$
coincides with the Lebesgue measure, while $P_M = P$.
Applying Pitman's result--an analogue of formula \eqref{b1}--gives
\begin{equation}
\label{pitman}
E[e^{-\alpha D + \beta G}] 
= 
\frac{\alpha E[e^{-\alpha D}] - \beta E[e^{-\beta D}] }{\alpha-\beta}.
\end{equation}
The joint Laplace transform of $D, G$ is thus expressible
in terms of the Laplace transform of $D$.
Combining the last two displays we obtain:
\[
E[e^{-\theta D} \mid Q_0=0] =
\frac{1}{P(Q_0=0)}
~
\frac{\theta^*}{\theta-\theta^*}
~
\left\{
E[ e^{-\theta^* D}] - E[ e^{-\theta D}]
\right\}.
\]
Using this in \eqref{almostthere} results in
\begin{multline}
\label{aye}
E[e^{-\alpha d(0)+\beta g(0)} \mid Q_0=0] 
=
\frac{1}{P(Q_0=0)}
~
\frac{\theta^*}{\alpha-\beta}
\\
\times
~
\left\{
\frac{\alpha}{\alpha-\theta^*}
\big( E[e^{-\theta^* D}] - E[e^{-\alpha D}] \big)
-
\frac{\beta}{\beta-\theta^*}
\big( E[e^{-\theta^* D}] - E[e^{-\beta D}] \big)
\right\}
\end{multline}
So far, the arguments are general and hold for both spectrally negative
and positive L\'evy processes $Y$, as long as $\theta^*$ is taken
as in \eqref{th*}.
Substituting next the expression for the Laplace transform of $D$
from \eqref{pikk6}, \eqref{pikk1} for the spectrally negative, respectively
positive, case, we obtain the result.
\qed

\subsection{Observed busy periods}
\label{sec6}
%Recall that a busy period of $Q$ starts at a point labelled $d(n)$ 
%and ends at $g(n+1)$. We are interested in the joint distribution of
%these two points, conditional on they event that a deterministic
%observation point $t_0$ is between them. We may take $t_0=0$, in which case,
%by the enumeration scheme of Section \ref{sec5}, we have
In this section we follow ideas in \cite{Si}.
We are interested in the distribution of the observed busy period,
as defined in \eqref{oBP}.
On the conditioning event $\{Q_0 >0\}$, we have, by our enumeration 
convention,
\[
g(0) < d(0) < 0 < g(1), \quad P-\text{a.s.}
\]
Using Lemma \ref{pikk2} with $N_1$ (resp.\ $N_2$)
the point process with points $\{d(n)\}$ (resp.\ $\{g(n\}$), we have
\begin{equation}
\label{b2}
E[e^{-\alpha g(1) + \beta d(0)} \mid Q_0 > 0]
= \frac{\alpha E[e^{-\alpha g(1)} \mid Q_0>0] - \beta E[e^{-\beta g(1)} \mid Q_0>0]  }
{\alpha-\beta},
\quad \alpha, \beta > 0.
\end{equation}
Recall the evolution equation for $Q$:
\begin{equation}
\label{recallq}
Q_t = Q_s + L(s,t]-(t-s) - \inf_{s \le u \le t} \{Q_s+L(s,u]-(u-s)\}.
\end{equation}
Let $s=0$ and assume $Q_0 > 0$. Since $X_0>0$, $P$-a.s., we have
$L(0,t]=0$ for all $0 < t < D = \inf\{r>0:X_r=0\}$, and so,
\[
Q_t = Q_0 - t - \inf_{0\le u \le t} \{Q_0-u\} = Q_0-t,
\quad
\text{a.s.\ on } 
\{Q_0 > 0,~ t < D\},
\]
which implies that
\begin{align}
\label{gcases}
g(1) =
\begin{cases}
Q_0,& \text{a.s.\ on } \{0< Q_0< D\},\\
g(1) \comp \theta_D , & \text{a.s.\ on } \{Q_0> D\}.
\end{cases}
\end{align}
Now, if $Q_0>D$, we have $Q_{D-} = Q_0-D$, so
{from} \eqref{recallq}, $Q$ evolves as
\[
Q_{D+t} = Q_0-D + L[D, D+t] - t,\quad t \ge 0,
\]
ans as long as $Q_{D+t} > 0$.
This implies that, a.s.\ on $\{Q_0 > D\}$,
\[
g(1) \comp \theta_D -D = \inf\{ t >0:~ Q_0-D + L[D, D+t] - t=0\}.
\]
Therefore \eqref{gcases} becomes
\begin{equation}
\label{gcases2}
g(1) =
\begin{cases}
Q_0,& \text{a.s.\ on } \{0< Q_0< D\},\\
D +  \inf\{ t >0:~ Q_0-D + L[D, D+t] - t=0\} , & \text{a.s.\ on } \{Q_0> D\}.
\end{cases}
\end{equation}
Consider now the inverse local time process, with the origin 
of time placed at $D$,
i.e.\
\[
L^{-1}_{D;x} := \inf\{ t > 0:~ L[D,D+t]> x\}, \quad x \ge 0.
\]
By the strong Markov property for $X$ at the stopping time $D$ we
have that the $P$-distribution of $(L^{-1}_{D;x}, ~ x \ge 0)$
is the same as the $P_0$-distribution of $(L^{-1}_x, ~ x \ge 0)$,
which has been identified in Propositions \ref{loc-} and \ref{loc+}:
%\[
%E [ e^{-q L^{-1}_{D;x}}] = 
%\begin{cases}
%e^{-x q/\Phi_Y(q)}, & \quad \text{ if $Y$ is spectrally negative},
%\\
%e^{-x\Phi_Y(q)}, & \quad \text{ if $Y$ is spectrally positive}.
%\end{cases}
%\]
Thus, $(L^{-1}_{D;x}, ~ x \ge 0)$ is a (proper) subordinator.
Consider next the spectrally negative L\'evy process
\[
\widetilde \Lambda_x:= x- L^{-1}_{D;x}, \quad x \geq 0.
\]
Notice that $P(\widetilde \Lambda_0=1)$.
The Laplace exponent of $\widetilde \Lambda$ is the function
$\psi_\Lambda$ of \eqref{psil}.
Define the hitting time of level $-a$ by $\widetilde \Lambda$
\[
\sigma(\widetilde \Lambda; a) := 
\inf\{x>0:~ \widetilde \Lambda_x < -a\}, \quad a> 0,
\]
Formula \eqref{tau-} gives us the Laplace transform of $\sigma(\widetilde \Lambda; a)$
in terms of the scale functions of $\widetilde \Lambda$, defined
in \eqref{LT} and \eqref{Ztransform}. Combining them, we obtain
\begin{equation}
\int_0^\infty e^{-\theta a}~ E [e^{-q \sigma(\widetilde \Lambda; a)}] ~da
= 
\frac{1}{\psi_\Lambda(\theta)-q}
\left( \frac{\psi_\Lambda(\theta)}{\theta} - \frac{q}{\Phi_\Lambda(\theta)}
\right) 
=: H^{(q)}(\theta).
\label{H}
\end{equation}
As can be easily seen from Lemma \ref{simple},
for any $a >0$,
\[
\inf\{t>0:~ t-L[D,D+t] \ge a\}
=
\inf\{x>0:~ \widetilde \Lambda_x < -a\} + a = \sigma(\widetilde \Lambda; a)+a.
\]
Using this in \eqref{gcases2}, we obtain
\[
g(1) =
\begin{cases}
Q_0,& \text{a.s.\ on } \{0< Q_0< D\},\\
Q_0 + \sigma(\widetilde \Lambda; Q_0-D), & \text{a.s.\ on } \{Q_0> D\}.
\end{cases}
\]
It is useful to keep
in mind that $\widetilde \Lambda$ is independent of $Q_0-D$, by
the strong Markov property of $X$ at $D$.
We are now ready to compute the Laplace transform appearing on the right
hand side of \eqref{b2}:
\begin{align}
E[e^{-\alpha g(1)}; Q_0>0]
&=
E[e^{-\alpha g(1)}; 0< Q_0< D]
+
E[e^{-\alpha g(1)}; Q_0>D]
\nonumber
\\
&= E[e^{-\alpha Q_0}; 0< Q_0< D]
+
E[e^{-\alpha (Q_0+\sigma(\widetilde \Lambda; Q_0-D)}; Q_0>D]
\nonumber
\\
&= E[e^{-\alpha Q_0}; Q_0>0] 
-
E[e^{-\alpha Q_0} ~(1-e^{-\alpha \sigma(\widetilde \Lambda; Q_0-D)}); Q_0>D]
\label{lotsofterms}
\end{align}
Recall that $P(Q_0 > x) = \mu e^{-\theta^* x}$, and so
 \begin{equation}
 E[e^{-\alpha Q_0}; \, Q_0 >0] = \mu \frac{\theta^*}{\alpha + \theta^*}.
 \label{firstterm}
 \end{equation}
 To compute the second and the third terms we need some elementary properties of exponentially distributed random variables which we state without proof.
 \begin{lemma} \label{pikk10}
   Let $T$ be an exponentially distributed r.v. with parameter $\lambda$
   and $(X,Y)$, $X \geq 0$, $Y\geq 0$, be a two-dimensional r.v. independent of $T$.
   Then $X$ and $T-X-Y$ are independent given $T>X+Y$. Moreover,
   $$E[e^{-\alpha (T-X-Y)} \, |\, T>X+Y]= \frac{\lambda}{\alpha + \lambda},$$
   $$E[e^{-\alpha X}; \, T>X+Y]= E[e^{-(\alpha + \lambda) X - \lambda Y}].$$
 \end{lemma}
%The next Lemma shows to which triple $(T,X,Y)$ the above lemma will be applied.
Use \eqref{recallq} once more with $s=G=\sup\{t < 0:~X_t=0\}$, and $t=0$,
taking into account the fact that $L$ is not
supported on $(G,0)$, to obtain  
\[
Q_0= Q_G + G, \quad \text{ a.s. on }
\{Q_0>0\}.
\]
Since $Q_G$ is exponentially distributed with parameter $\theta^*$ and independent of $(G, D)$ 
(from Lemma \ref{3indep}), we have, applying Lemma \ref{pikk10}, the following result:
 \begin{lemma} \label{pikk11}
   Given $Q_0>D$, the r.v.'s $Q_0-D$ and $D$ are independent. Moreover,
   $$E[e^{-\alpha(Q_0-D)}\, |\, Q_0>D]= \frac{\theta^*}{\alpha + \theta^*},$$
   $$E[e^{-\alpha D}; \, Q_0>D] = E[e^{-(\alpha + \theta^*)D + \theta^*G}].$$
 \end{lemma}
Using Lemma $\ref{pikk11}$, we write the last term of 
\eqref{lotsofterms} as follows:
\begin{align}
E[e^{-\alpha Q_0} &~(1-e^{-\alpha \sigma(\widetilde \Lambda; Q_0-D)}); Q_0>D]
\nonumber \\
&=
P(Q_0>D) ~
E[e^{-\alpha D} ~ e^{-\alpha (Q_0-D)} ~
(1-e^{-\alpha \sigma(\widetilde \Lambda; Q_0-D)}) \mid Q_0>D]
\nonumber \\
&=P(Q_0>D) ~
E[e^{-\alpha D} \mid Q_0>D]
~
E[ e^{-\alpha (Q_0-D)} ~ (1-e^{-\alpha \sigma(\widetilde \Lambda; Q_0-D)}) \mid Q_0>D]
\nonumber \\
&=
E[e^{-(\alpha + \theta^*)D + \theta^*G}]
~
E[ e^{-\alpha V} ~ (1-e^{-\alpha \sigma(\widetilde \Lambda; V)})],
\label{thisandthat}
\end{align}
where, in the last term, we introduced a random variable $V$, exponentially 
distributed with parameter $\theta^*$, independent of everything else
(due to the fact that $Q_0-D$, conditionally on being positive, is exponential
with parameter $\theta^*$, independent of $\widetilde \Lambda$).
The first term of \eqref{thisandthat} can be computed as in \eqref{pitman}.
We have, for all $\alpha, \beta \ge 0, \quad \alpha \ne \beta$,
\begin{equation*}
\label{pikk7}
E [e^{-\alpha D + \beta G}]  = \frac{\Phi_Y(0)}{\alpha -\beta} \,\Big( \frac{\alpha}{\Phi_Y(\alpha)}
      - \frac{\beta}{\Phi_Y(\beta)}\Big),
  \end{equation*}
and for the spectrally positive case,
  for all $\alpha, \beta \ge 0, \quad \alpha \ne \beta$,
\begin{equation*}
\label{pikk17}
E [e^{-\alpha D + \beta G}]
 = \psi_Y'(0+) \frac{\Phi_Y(\alpha) - \Phi_Y(\beta)}{\alpha - \beta}.
  \end{equation*}
Note that taking account of the definition \eqref{psil}
of $\psi_\Lambda$ for both the spectrally negative and positive cases, 
and the fact that $\psi_\Lambda(\theta^*)=0$,
one sees that generically for both cases, 
 for all $\alpha \ge 0$,
\begin{equation}
\label{pikk-andreas}
E [e^{-(\alpha +\theta^*) D + \theta^* G}] 
 =   \frac{\mu}{\alpha} (\alpha - \psi_\Lambda(\alpha+\theta^*)).
  \end{equation}
For the second term of \eqref{thisandthat} we have, using \eqref{H},
\begin{align}
E[ e^{-\alpha V} (1-e^{-\alpha \sigma(\widetilde \Lambda; V)})]
&= 
\frac{\theta^*}{\alpha+\theta^*} 
-
\theta^* \int_0^\infty e^{-\theta^* v} e^{-\alpha v} 
E[e^{-\alpha \sigma(\widetilde \Lambda; v)}] dv
\nonumber
\\
&=
\frac{\theta^*}{\alpha+\theta^*}
-
\theta^* H^{(\alpha)}(\alpha+\theta^*)
\nonumber
\\
&=
\frac{\theta^*}{\alpha+\theta^*}
-
\frac{\theta^*}{\psi_\Lambda(\alpha+\theta^*)-\alpha}
\left(
\frac{\psi_\Lambda(\alpha+\theta^*)}{\alpha+\theta^*}
-\frac{\alpha}{\Phi_\Lambda(\alpha)}
\right)
\nonumber
\\
&= \frac{\theta^*}{\alpha+\theta^*} ~
\frac{\alpha+\theta^*-\Phi_\Lambda(\alpha)}{\Phi_\Lambda(\alpha)}~
\frac{\alpha}{\psi_\Lambda(\alpha+\theta^*)-\alpha} 
\label{pikk-takis}
\end{align}
Multiplying \eqref{pikk-andreas} and \eqref{pikk-takis}
we obtain the following expression for \eqref{thisandthat}:
\[
E[e^{-\alpha Q_0} (1-e^{-\alpha \sigma(\widetilde \Lambda; Q_0-D)}); Q_0>D]
=
\frac{\mu\theta^*}{\alpha+\theta^*}
-\frac{\mu\theta^*}{\Phi_\Lambda(\alpha)}
\]
This, together with \eqref{firstterm} and \eqref{lotsofterms}, yields:
\[
E [e^{-\alpha g(1)} \mid Q_0>0] = \frac{\theta^*}{\Phi_\Lambda(\alpha)}.
\]
Using \eqref{b2},
we finally come to rest at the following main result. 
\begin{proposition}[distribution of observed busy period]
 \label{pikk12}
  For $\alpha, \beta >0$, $\alpha \ne \beta$,
  $$E[e^{-\alpha g(1) + \beta d(0)} \, |\, Q_0>0] = \frac{\theta^*}{\alpha - \beta}
  \Big( \frac{\alpha}{\Phi_\Lambda(\alpha)} - \frac{\beta}{\Phi_\Lambda(\beta)} \Big)$$
  where $\Phi_\Lambda$ is the right inverse of $\psi_\Lambda$ which is given in (\ref{psil}) and (cf. Theorem \ref{distr0})
\begin{description}
\item[(i)] $\theta^*>0$ is defined by $\psi_Y(1)$ in the case that $Y$ is spectrally negative,
\item[(ii)]  $\theta^*>0$ is defined by $\psi_Y(\theta^*) =\theta^*$  in the case that $Y$ is spectrally positive.
 \end{description}
\end{proposition}

\subsection{Typical idle and busy periods}
We now consider the problem of identifying the distribution of a typical
idle and a typical busy period of $Q$. We place the origin of time at
the beginning of such a period, by considering the appropriate Palm 
probability. Let $N_g$ (resp.\ $N_d$)
be the point process with points $\{g(n)\}$ (resp.\ $\{d(n)\}$), the
beginnings of idle (resp.\ of busy) periods, 
and let $P_g$ (resp.\ $P_d$) be the Palm probability with respect
to $N_g$ (resp.\ $P_d$).
\begin{figure}[h]
\begin{center}
\epsfig{file=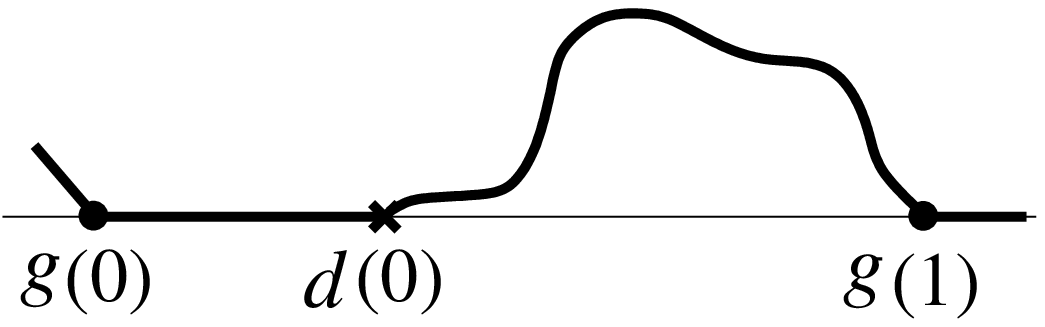,width=5cm}
\end{center}
\end{figure}
\\
Using \eqref{aaa} we have
\[
\frac{1-\mu}{\lambda} \alpha E[e^{-\alpha d(0)} | Q_0=0]
= 1- E_g[e^{-\alpha d(0)}],
\]
where $\mu=P(Q_0>0)$ and $\lambda$ is the common rate of $N_g$ and $N_d$.
The right side is precisely what we need. 
Everything in the left side is known (see Prop.\ \ref{pikk8}) except
the rate $\lambda$. Consider first the case when $Y$ is spectrally negative. 
Using Prop.\ \ref{pikk8}(i) with $\beta=0$ we get
\[
\frac{1-\mu}{\lambda} 
\frac{\Phi_Y(0)}{1-\Phi_Y(0)} \psi_Y(1) \frac{\alpha}{\alpha-\psi_Y(1)}
\frac{\Phi_Y(\alpha)-1}{\Phi_Y(\alpha)}
= 1- E_g[e^{-\alpha d(0)}].
\]
Taking limits as $\alpha \to \infty$--and since 
$\Phi_Y(\alpha) \to \infty$--we find the value of 
$\lambda$ and so the Laplace transform 
$E_g[e^{-\alpha d(0)}]$ of the typical idle period.
The result is in Proposition \ref{pikk13}(i) below.

We repeat the procedure for the spectrally positive case
and, using Proposition \ref{pikk8}(ii), we obtain:
\[
\frac{1-\mu}{\lambda} 
\frac{\psi_Y'(0)}{1-\psi_Y'(0)}\, \theta^*\, 
\frac{\alpha-\Phi_Y(\alpha)}{\alpha-\theta^*}
= 1- E_g[e^{-\alpha d(0)}].
\]
Note that $\lim_{\theta \to \infty} \psi_Y(\theta)/\theta = \infty$ if
$Y$ is of unbounded variation, and so
$\lim_{\alpha \to \infty} \Phi_Y(\alpha)/\alpha = 0$.
Thus, we can find $\lambda$ and $E_g[e^{-\alpha d(0)}]$--see
Proposition \eqref{pikk13}(ii) below.

But if $Y$ is spectrally positive (with non-monotone paths) and
of bounded variation then $L$ is absolutely continuous
and has a drift $d_Y$--see \eqref{abs-cts}.
We can easily see, e.g.\ from \eqref{Ybv}, that
\[
\psi_Y(\theta) = \log E [e^{-\theta (Y_1-Y_0)}]
= |d_Y| \theta - \int_0^\infty (1-e^{-\theta y}) \Pi(dy),
\]
and, since $\int_0^\infty (y \wedge 1) \Pi(dy) < \infty$, we 
obtain $\lim_{\theta \to \infty} \psi_Y(\theta)/\theta = |d_Y|$.
So $\lim_{\alpha \to \infty} \Phi_Y(\alpha)/\alpha = 1/|d_Y|$.
Again, we can find $\lambda$ and $E_g[e^{-\alpha d(0)}]$--see
Proposition \eqref{pikk13}(iii) below.

\begin{proposition}[distribution of typical idle period]
\label{pikk13}
Fix $\alpha > 0$. Let $P_g$ be Palm probability 
with respect to the beginnings of idle periods of $Q$. 
\begin{description}
\item[(i)] When $Y$ is spectrally negative we have
\[
\lambda = \Phi_Y(0) \psi_Y(1),
\qquad
E_g[e^{-\alpha d(0)}] = 1 - \frac{\alpha}{\Phi_Y(\alpha)}
~
\frac{\Phi_Y(\alpha)-1}{\alpha-\psi_Y(1)}.
\]
\item[(ii)] When $Y$ is spectrally positive and $\int_0^1 y \Pi(dy)=\infty$,
we have
\[
\lambda = \psi_Y'(0) \theta^*,
\qquad
E_g[e^{-\alpha d(0)}] = 1 - \theta^*~
\frac{\alpha-\Phi_Y(\alpha)}{\alpha-\theta^*},
\]
where $\theta^*>0$ satisfies $\theta^*=\psi_Y(\theta^*)$.
\item[(iii)] When $Y$ is spectrally positive and $\int_0^1 y \Pi(dy)<\infty$,
we have
\[
\lambda = \psi_Y'(0) \theta^* \left(1-\frac{1}{|d_Y|}\right),
\qquad
E_g[e^{-\alpha d(0)}] = 1 - \frac{\theta^*}{1-\frac{1}{|d_Y|}}~
\frac{\alpha-\Phi_Y(\alpha)}{\alpha-\theta^*},
\]
where $\theta^*>0$ satisfies $\theta^*=\psi_Y(\theta^*)$,
and $d_Y$ is the drift defined in \eqref{Ybv}-\eqref{abs-cts}.
\end{description}
\end{proposition}
\noindent
\begin{remark} {\rm By Assumption [A2]--see Section \ref{sec3}--we
have $d_Y < -1$ and so the constant above is positive.}
\end{remark}

In the same vein, we obtain the Laplace transform 
$E_d[e^{-\alpha g(1)}]$ of a typical busy period. Notice that,
under $P_d$, we  have $d(0) = 0$, and so the first
busy period to the right of the origin of time
is the interval $(d(0), g(1))$.
We have,
\[
\frac{\mu}{\lambda} \alpha E[e^{-\alpha g(1)}| Q_0>0]
= 1- E_d[e^{-\alpha g(1)}].
\]
Using Proposition \ref{pikk12} with $\beta=0$, we have
\[
 \frac{\mu\theta^*}{\lambda} \frac{\alpha}{\Phi_\Lambda(\alpha)}
= 1- E_d[e^{-\alpha g(1)}].
\]
{From} the expression \eqref{psil} for $\psi_\Lambda(\theta)$ we find
$\lim_{\theta\to\infty} \psi_\Lambda(\theta)/\theta = 1$.
So, $\lim_{\alpha\to\infty} \alpha/\Phi_\Lambda(\alpha) = 1$.
Therefore:
\begin{proposition}[distribution of typical busy period]
\label{pikk14}
Fix $\alpha > 0$. Let $P_d$ be Palm probability
with respect to the beginnings of busy periods of $Q$.
Let $\psi_\Lambda$ be defined as in \eqref{psil}, and let
$\Phi_\Lambda$ be its right inverse function.
Then
\[
\lambda = \mu \theta^*,
\qquad 
E_d[e^{-\alpha g(1)}] = 1 - \frac{\alpha}{\Phi_\Lambda(\alpha)},
\]
where $\mu = \Phi_Y(0)$, $\theta^* = \psi_Y(1)$, if $Y$ is spectrally negative;
and $\mu = \psi_Y'(0)$, $\theta^*>0$ is defined through 
$\theta^*=\psi_Y(\theta^*)$, if $Y$ is spectrally positive.
\end{proposition}

\begin{corollary}
The mean duration of a typical idle period is $(1-\mu)/\lambda$,
while the mean duration of a typical busy period is $\mu/\lambda$,
where $\mu$ is given by \eqref{ltrate} and $\lambda$ is given in
Propositions \ref{pikk13}, \ref{pikk14}.
\end{corollary}

\begin{remark}
{\rm
By the relation between $P$ and the Palm probability $P_d$ it
follows that the typical idle (respectively busy)
periods are stochastically smaller than 
the observed idle (respectively busy) periods, see \cite{BB}. 
In particular, the means of the former are shorter than
the means of the latter. (This is usually referred to as 
the ``inspection paradox''.)
This gives us several inequalities between different quantities associated with 
the L\'evy process $Y.$ To give an example, we compare the mean durations of idle periods in the spectrally negative 
case. From Proposition \ref{pikk8} we have  
\begin{align*}
  E \big[ e^{-\alpha d(0)} \mid Q_0=0\big]  
&= \frac{\Phi_Y(0)\psi_Y(1)}{1-\Phi_Y(0)}~
   \left(
\frac{\Phi_Y(\alpha)-1}{ \Phi_Y(\alpha)(\alpha - \psi_Y(1))}
\right)
\\
&=:\frac{\Phi_Y(0)\psi_Y(1)}{1-\Phi_Y(0)}~F(\alpha).
\end{align*}
It follows that 
 $$
 E \big[ d(0) \mid Q_0=0\big]  =-\frac{\Phi_Y(0)\psi_Y(1)}{1-\Phi_Y(0)}~F'(0).
$$
Since $d(0)$ and $-g(0)$ are identical in law, the mean duration of the observed idle period is 
 $$
 E \big[ d(0)-g(0) \mid Q_0=0\big]  =-2\,\frac{\Phi_Y(0)\psi_Y(1)}{1-\Phi_Y(0)}~F'(0).
$$
Notice that, using the function $F,$ we may write from Proposition \ref{pikk13}
$$
E_g[e^{-\alpha d(0)}] = 1 - {\alpha}\,F(\alpha). 
$$
and
$$
E_g[ d(0)] =F(0).
$$
Now, by the ``inspection paradox'', 
$$
E_g[ d(0)] \leq E \big[ d(0)-g(0) \mid Q_0=0\big],
$$
which, after some manipulations, is equivalent with 
$$
\left(1-\Phi_Y(0)\right)^2\leq 2\left(\Phi_Y(0)^2-\Phi_Y(0)+\Phi'_Y(0)\psi_Y(1)\right).
$$}
\end{remark}

\noindent{\bf Example 5 (continuation of Example 1):}
Consider $Y_t=\sigma B_t-\mu t$, and assume $0 < \mu < 1$.
%and take, for simplicity, $\sigma=1$.
Here the rate of beginnings of idle (or busy) periods is
\[
\lambda = \psi_Y'(0) \theta^* = \frac{2 \mu (1-\mu)}{\sigma^2}.
\]
The mean duration of a typical idle period of $Q$ is
\[
\frac{\sigma^2}{2\mu},
\]
while the mean duration of a typical busy period of $Q$ is
\[
\frac{\sigma^2}{2(1-\mu)}.
\]
To find, e.g., the distribution of a typical busy period, we use
Proposition \ref{pikk14}. We have, see \eqref{psil},
\[
\psi_\Lambda(q) = q - \Phi_Y(q),
\]
where $\Phi_Y$ is the inverse function of $\psi_Y$, i.e.
\[
\Phi_Y(q) = \frac{\sqrt{\mu^2+2\sigma^2 q^2}-\mu}{\sigma^2},
\]
and $\Phi_\Lambda$ is the inverse function of $\psi_\Lambda$, i.e.
\[
\Phi_\Lambda(\alpha) 
=
\frac{(1-\mu) + 2 \sigma^2 \alpha 
+ \sqrt{(1-\mu)^2 + 4 \sigma^2 \alpha } }{2\sigma^2},
\]
and so the Laplace transform of the typical busy period is
\[
E_d[e^{-\alpha g(1)}]
= \frac{(1-\mu)+\sqrt{(1-\mu)^2 + 4 \sigma^2 \alpha }}{(1-\mu)+2\sigma^2\alpha+\sqrt{(1-\mu)^2 + 4 \sigma^2 \alpha }}.
\]

\appendix%{Appendix A: On Skorokhod reflection, fluid queues, and stationarity}

\section{On Skorokhod reflection, fluid queues, and stationarity}
\label{skororeview}
In this section we review some facts about the Skorokhod reflection
of a process with stationary-ergodic increments. 
We carefully define the system, give conditions for its stability,
%discuss an alternative integral representation,
and recall some distributional relations based on Palm calculus,
see \cite{BB}, \cite{KL}. 
Although the setup is much more general than the one used in later 
sections for concrete calculations, it is nevertheless interesting to
isolate those properties that are not based on specific distributional
assumptions (such as Markovian property or independent increments)
but are consequences the more general stationary framework. 

Let $(\Omega, \FF, P)$ be a probability space together with a
$P$-preserving flow $(\theta_t, t \in \R)$. That is, for each $t \in \R$,
$\theta_t : \Omega\to \Omega$ is measurable with measurable
inverse, $\theta_0$ is the identity function,
$\theta_t \comp \theta_s = \theta_{s+t}$, for all $s, t \in \R$,
and $P(\theta_t A)=P(A)$ for all $t \in \R$, $A \in \FF$.
Consider a process $\WW=(\WW_t, t \in \R)$ with stationary
increments, i.e.\
$(\WW_t-\WW_s) \comp \theta_u = \WW_{t+u}-\WW_{s+u}$
for all $s, t, u \in \R$. 
We let $E$ denote expectation with respect to $P$.

Following \cite{KL}, we define the 
{\em ``Skorokhod Dynamical System''
(abbreviated SDS henceforth) driven by $\WW$} as a 2-parameter stochastic flow:
\begin{align}
\RR_{s,t}\WW(x) 
&:= [x+\WW_t-\WW_s] - \inf_{s \le u \le t} [(x+W_s-W_u) \wedge 0]
\label{sds}
\\
&:= \sup_{s \le u \le t} (\WW_t - \WW_u) \vee (x+\WW_t-\WW_s)
\quad x \ge 0, \quad s \le t.
\label{sds2}
\end{align}
Thus, for each $s < t$, we have a random element $\RR_{s,t}\WW$ taking
values in the space of $C(\R_+)$ of continuous functions from $\R_+$
into itself. The family $(\RR_{s,t}\WW,~ -\infty < s < t < \infty)$ is
a stochastic {\em flow} because the 
following composition rule (semigroup property)
holds for each $\omega \in \Omega$:
\begin{align*}
& \RR_{s,t}\WW = \RR_{u,t}\WW \comp \RR_{s,u}\WW,
\quad s \le u \le t,\\
& \RR_{t,t}\WW(x) = x, \quad t \in \R, x \ge 0,
\end{align*}
%Physically, for each $t_1 < t_2$ we may think of $\WW_{t_2}-\WW_{t_1}$ 
%as representing the net flow between the times $t_1$ and $t_2$.
%Fixing time $s$ and a buffer content $x \ge 0$, the variable
%$\RR_{s,t}\WW(x)$ represents the buffer content at some time $t > s$.
It is a {\em stationary} stochastic flow because,
for each $x \in \R+$, we have:
\[
\RR_{s,t}\WW(x)\comp \theta_u
= \RR_{s+u,t+u}\WW(x),~ -\infty < s \le t < \infty, ~ u \in \R.
\]
%but, for each fixed $s$, the process $\big\{ \RR_{s,t}\WW(x),~ t > s\big\}$
%is not (necessarily) stationary.
%\subsection{Stationary reflection}
We say that the process $\ZZ=\big\{ \ZZ_t, t \in \R \big\}$ constitutes
a stationary solution of the SDS driven by $\WW$ 
if $\ZZ$ is $\WW$-measurable and if
\begin{align*}
\ZZ_t &= \RR_{s,t}\WW(\ZZ_s), \quad s \le t,\\
\ZZ_t\comp \theta_u &=\ZZ_{t+u}, \quad t, u \in \R.
\end{align*}
Existence and uniqueness is guaranteed under some assumptions:
\begin{lemma}\label{u1}
Assume that $\sup_{-\infty < s \le 0} \WW_s < \infty$,
and $\liminf_{t\to\infty} \WW_t < \infty$, $P$-a.s.
Then there is a unique stationary solution to the Skorokhod dynamical
system driven by $\WW$. This is given by
\begin{equation}
\label{qq}
\ZZ_t = \sup_{-\infty < u \le t} (\WW_t-\WW_u) 
=: \widetilde \RR_t \WW.
\end{equation} 
\end{lemma}

Quite often, in addition to stationarity of the flow, we also
also assume ergodicity, namely that each
$A \in \FF$ that is invariant under $\theta_t$ for all $t$, 
has $P(A)$ equal to $0$ or $1$. 
Owing to Birkhhoff's individual ergodic theorem Lemma \ref{u1}
immediately yields:
\begin{corollary}\label{u2}
Under the ergodicity assumption, and if $E \WW_1< 0$, then
there is a unique stationary solution $\ZZ$ to the Skorokhod dynamical
system driven by $\WW$. The process $\ZZ$ is given by \eqref{qq}
and $\ZZ$ is an ergodic process.
\end{corollary}
%\subsection{Reflecting Brownian motion with drift}
%As an example, consider a 2-sided standard
%Brownian motion $B=(B_t, t\in \R)$, and a $\mu > 0$, and let $\WW_t=B_t-\mu t$.
%Let $X^x = (X^x_t, t \ge 0)$ be the solution to the SDS driven by $\WW$
%and starting from $x$ at time $0$:
%\[
%X^x_t := \RR_{0,t}\WW(x), \quad t \ge 0.
%\]
%This is a reflecting Brownian motion with drift. 
%To check that the conditions of
%Corollary \ref{u2} hold, simply let $\Omega$ be the canonical
%space of realisations of $B$ (the space of continuous functions
%$\omega : \R \to \R$ with $\omega(0)=0$), 
%let $P$ be the Wiener measure on the Borel sets of $\Omega$,
%let $\theta_t\omega(\cdot) := \omega(t + \cdot)$, 
%let $B_t(\omega)\equiv\omega(t)$, and
%recall that stationarity and ergodicity hold, 
%and that $E\WW_1 = -\mu < 0$.
%We then have that the stationary solution $X=(X_t, t \in \R)$ to
%the SDS driven by $\WW$ is given by
%\[
%X_t = \sup_{-\infty < u \le t} (B_t-B_u-\mu t+\mu u).
%\]
%Note that $X$ is a Markov process. Its marginal is well-known:
%\[
%P(X_t > x) = e^{-2\mu x}.
%\]
%\subsection{Bounded variation paths}
For the purposes of this paper, assume that $\WW$ 
is of the form
\begin{equation}
\label{Wdecomp}
\WW_t-\WW_s=A(s,t]-\beta (t-s), \quad s \le t ,
\end{equation}
where $A$ is a  locally finite stationary random measure, and
\begin{equation}
\label{Wdecompa}
0 < \beta < \alpha := E A(0,1).
\end{equation}
Let $P_A$ be the Palm probability, see \cite{BB}, \cite{KAL}, 
with respect to $A$:
%%In other words, for each $C \in \FF$,
%%$P_A(C)$ is the Radon-Nikod\'ym derivative of
%%the Campbell measure $E[\1_C A(dt)]$ with respect to $E[A(dt)]=\alpha dt$.
%%Due to stationarity, we also have \cite{BB}
\[
P_A(C) = \frac{1}{\alpha} E\bigg[ \int_{(0,1]} \1_C \comp \theta_t A(dt)
\bigg].
\]
%%We may formally take the latter as the definition of the probability
%%measure $P_A$ on $(\Omega, \FF)$.
%%We next derive some distributional relations.
%%We specialise to the case where $A$ is stationary-ergodic with rate $\alpha$,
%%and $B$ is a multiple of the Lebesgue measure: 
%%\[
%%B(s,t] = \beta(t-s), \quad
%%0 < \beta < \alpha < \infty.
%%\]
%The following 
%``distributional Little's 
%law for fluid queues''\footnote{The terminology is 
%suggestive and borrows freely from
%analogies from ordinary Queueing Theory.}
%was proved in \cite{KL}: 
%Note that $P_A(\ZZ_0>0)=1$.
The following is a consequence of Theorem 3 of \cite{KL}:
\begin{lemma}[distributional Little's law]
Let $\ZZ$ be the unique stationary solution to the SDS driven by
$\WW$ of the form \eqref{Wdecomp}. Assume that \eqref{Wdecompa} holds.
Then, for any function $\psi:[0,\infty) \to \R$, which
is continuous on $(0,\infty)$, we have (Theorem 3 of \cite{KL})
\[
E\psi(\ZZ_0) = \left(1-\frac{\alpha}{\beta} \right)\psi(0)
+ \frac{\alpha}{\beta} E_A \psi(\ZZ_0).
\]
In particular, %%for the choice $\psi(x) = \1(x > 0)$, we have
\begin{equation}
\label{dll}
P(\ZZ_0 > x) = \frac{\alpha}{\beta} P_A(\ZZ_0 > x), \quad x > 0,
\qquad
P(\ZZ_0 > 0) = \frac{\alpha}{\beta}.
\end{equation}
\end{lemma}
%and this may be called ``Little's law for the server''.
%Integrating \eqref{dll} over $x$, we obtain
%\[
%E[\ZZ_0] = \frac{\alpha}{\beta} E_A[\ZZ_0].
%\]
%and this may be called ``Little's law for fluid queues'' as it expresses, 
%mathematically, the natural identity that the mean amount of fluid in
%the system ($=E[\ZZ_0]$) is equal to the arrival rate ($=\alpha$)
%times the mean amount of time spent by a typical ``customer''
%in the system ($=E_A[\ZZ_0]/\beta$).
%
%The case of interest in this paper is $A=L=$ the local time
%at a specific point of a Markov process $X$.
%Let $X=\{X_t, t \in \R\}$ be a stationary Markov process. 
%Assume that it has, at a specific point
%of its state space, a local time $L$. Since $X$ is stationary, $L$
%is a process with stationary increments. Therefore $\WW_t:= L_t-t$
%has stationary increments and we can view the SDS driven by $\WW$ as
%a special case of the discussion above with $A$ being the random measure
%driven by $L$ and $B$ being the Lebesgue measure.
%Mannersalo {\em et al.} studied the case where $X_t$ is the process
%obtained by reflecting a Brownian motion with drift 
%$\mu$ and $L$ is its local time at the point $0$. 
%We shall extend the results to the case where $X$ is obtained by 
%reflecting a spectrally negative L\'evy process.
It should be noted that the decomposition \eqref{Wdecomp} 
of $W$ is not unique; nevertheless,
\eqref{dll} holds, regardless of which decomposition of $W$ we choose.

\section{Exit times for spectrally negative L\'evy processes}
\label{standard}

In this section we consider a spectrally negative L\'evy 
process and some facts regarding the first time the
process exits an unbounded interval.
Let $Y=(Y_t, t \in \R)$ be a spectrally negative L\'evy process
and L\'evy measure $\Pi$. 
In other words, let $B$ be a standard Brownian motion,
$\eta$ an independent Poisson random measure on $\R\times \R_-$ such that
\begin{equation} \label{prm}
E \eta (dt, dy) = dt \Pi(dy),
\quad
\Pi\{0\} = 0,
\quad
\int_{\R_-} (y^2 \wedge 1) \Pi(dy) < \infty,
\end{equation}
let $a \in \R$, $\sigma \ge 0$,
and define, for $-\infty< s \le t < \infty$,
\begin{multline}
\label{levyito}
Y(s,t] =
a (t-s) + \sigma (B_t-B_s)
\\
+\int_{(s,t]} \int_{(-\infty, -1]} y~\eta(du, dy)  +
\int_{(s,t]} \int_{(-1,0)} y~[\eta(du, dy) - du \Pi(dy)].
\end{multline}
Notice that we have thus defined only the increments of $Y$; the exact
value of $Y_0$ is unimportant;
we may, arbitrarily, set 
\[
Y_0=0.
\]
(The reason that increments are more fundamental than the process itself
is amply explained in Tsirelson \cite{TSIR}.)
If we set
\[
Y_t :=
\begin{cases}
Y(0,t], & t \ge 0 \\
-Y(t,0], & t <0
\end{cases}, \quad t \in \R,
\]
we have
\[
Y(s,t] = Y_t-Y_s,
\quad -\infty< s \le t < \infty.
\]
(In case that $\int_{-1}^0 |y| \Pi(dy) < \infty$, and $\sigma=0$,
the process $Y$ has bounded variation paths and 
can also represented as
\begin{equation}
\label{levyitobv}
Y(s,t] =
d_Y (t-s) +\int_{(s,t]} \int_{(-\infty, 0]} y~\eta(du, dy) ,
\end{equation}
for some constant $d_Y$ known as the drift of $Y$.)

To be more precise, especially for the construction of the stationary versions
of processes in this paper, we introduce shifts.
Assume that $(B, \eta)$ is defined on a probability space $(\Omega, \FF, P)$
taken, without loss of generality, to be the canonical space 
$\Omega = C(\R) \times \NN(\R^2)$, where $C(\R)$ are the
continuous functions on $\R$, and $\NN(\R^2)$ are the integer-valued
measures $\R^2$. 
Let $P$ be the product measure on the
Borel sets\footnote{The space $\NN(\R^2)$ is endowed with the topology
of weak convergence; see, e.g., \cite{KAL1}.} of 
$C(\R) \times \NN(\R^2)$ that makes $B$ a standard Brownian motion
and $\eta$ a Poisson random measure with mean measure as in \eqref{prm},
and to each $\omega=(\phi, \mu) \in C(\R) \times \NN(\R^2)$,
let $B(\phi, \mu) = \phi$, $\eta(\phi, \mu) = \mu$.
Consider also the natural shift 
$(\theta_t, t \in \R)$ on $\Omega$ defined by
\[
\theta_t (\phi, \mu)(s, A) =
(\phi(t+s), \mu(A+s)), \quad
s \in \R, \quad A \in \BB(\R^2),
\]
where $A+s := \{(t+s, y) \in \R^2: (t,y)\in A\}$.
By construction, $Y$ has \cadlag paths, and, under $P$, it has 
stationary (and independent) increments.
Henceforth we shall denote by $P_x$  the conditional probability of 
$P$ given $Y_0=x$ and $E_x$ expectation with respect to it.
All of the
following facts are standard results which can be found, for example,
in \cite{BER,K} 
See also \cite{KP} for a
review which is more convenient for the setting at hand.

Let $\Psi_Y: \R \to \C$ denote the characteristic exponent of $Y$:
\[
E \big[ e^{i \theta Y_1} \big] = e^{-\Psi_Y(\theta)},
\quad \theta \in \R,
\]
and let $\psi_Y: [0,\infty) \mapsto \R$ denote the Laplace
exponent of $Y$:
\[
\psi_Y(\beta) = \log E \big[ e^{\beta Y_1} \big],
\quad \beta\geq 0.
\]
It is well known
that $\psi_Y$ is infinitely differentiable, strictly convex, $\psi(0)=0$, 
$\lim_{\beta \to \infty} \psi_Y(\beta) = \infty$,
and
\[
\psi'_Y(0+) = E Y_1 = E(Y_{t+1}-Y_t) \in \R \cup \{-\infty\}.
\]  
For each $q\geq 0$ let
\begin{equation}
\label{phi}
\Phi_Y(q) = \sup\{\beta \geq 0 : \psi_Y(\beta) = q\}.
\end{equation}
Since $Y$ drifts to infinity, oscillates, drifts to minus infinity
accordingly as $\psi'_Y(0+)>0$, $\psi'_Y(0+)=0$ and $\psi'_Y(0+)<0$, 
it follows that $\Phi_Y(0)>0$ if and only if $\psi'_Y(0)<0$ and otherwise
$\Phi_Y(0)=0$. 
It is also easy to see that $\Phi_Y(q) >0$ for all $q>0$. 

%The quantity $\Phi_Y(q)$ also turns out to characterize the first passage times of the process $Y$. Specifically,  if we define 
%$$
%\tau^+_x = \inf\{t>0 : Y_t >x\}
%$$ 
%then for all $q\geq 0$, $x\geq 0$ it is known that 
%\begin{equation}
%E_0 \big[ e^{-q \tau^+_x} \big] = e^{- \Phi_Y(q)x}.
%\label{tau+} 
%\end{equation}

%%%%%%%%%%%%%%
Define also the scale functions $W^{(q)}(x)$, $Z^{(q)}(x)$
via their Laplace transforms 
\begin{align}
\int_0^\infty e^{-\beta x}W^{(q)}(x)dx & = \frac{1}{\psi_Y(\beta) - q}, 
%\quad \beta > \Phi_Y(q), 
\label{LT} \\
\int_0^\infty e^{-\beta x} Z^{(q)}(x)dx
& = \frac{\psi_Y(\beta)}{\beta(\psi_Y(\beta) - q)},
\label{Ztransform}
\end{align}
defined for all $\beta > \Phi_Y(q)$.
%via the relations
%\begin{align}
%\label{LT}
%\int_0^\infty e^{-\beta x}W^{(q)}(x)dx = \frac{1}{\psi_Y(\beta) - q}, &
%\quad
%\beta > \Phi_Y(q), \\
%Z^{(q)}(x) = 1 + q\int_0^x W^{(q)}(y)dy.&
%\nonumber
%\end{align}
%A straightforward calculation shows that for each $q\geq 0$,
%\begin{equation}
%\label{Ztransform}
%\int_0^\infty e^{-\beta x} Z^{(q)}(x)dx
%= \frac{\psi_Y(\beta)}{\beta(\psi_Y(\beta) - q)},
%\quad
%\beta > \Phi_Y(q).
%\end{equation}
The functions $\Phi_Y, W^{(q)}, Z^{(q)}$ 
appear in the expressions for
the Laplace transform of the first passage times 
\begin{align}
\tau^+_x &:= \inf\{t>0 : Y_t >x\},
\label{t+}\\
\tau^-_{-x} &= \inf\{t >0 : Y_t <-x\},
\label{t-}
\end{align}
as follows (cf. Chapter Theorem 8.1, p214 of Kyprianou \cite{K}):
\begin{lemma}
\label{one-sided-exit}
For all $q\geq 0$, $x \geq 0$,
\begin{align}
E \big[ e^{-q \tau^+_x} \big] &= e^{- \Phi_Y(q)x},
\label{tau+} \\
E \big[ e^{-q \tau^-_{-x}} \big] &= Z^{(q)}(x) - \frac{q}{\Phi_Y(q)}W^{(q)}(x).
\label{tau-}
\end{align}
\end{lemma}


\begin{thebibliography}{99}
\bibitem{BB}
Baccelli, F., and Br\'emaud, P. (2003)
{\em Elements of Queueing Theory: 
Palm Martingale Calculus and Stochastic Recurrences,}
2nd edition, Springer.
\bibitem{BER}
Bertoin, J. (1996)
{\em L\'evy processes}.
Cambridge University Press.
\bibitem{BIN}
Bingham, N.H. ((1975)
Fluctuation theory in continuous time.
{\em Adv.\ Appl.\ Prob.} {\bf 7}, 705-766.
\bibitem{BS}
Borodin, A.N., Salminen, P. (1996)
{\em Handbook of Brownian Motion - Facts and Formulae.}
Birkh\"auser.
\bibitem{Fri}
Fristedt, B.E. (1974) Sample functions of stochastic processes with stationary independent increments. {\em Adv. \ Probab.} {\bf 3}, 241--396. Dekker, New York
\bibitem{KAL}
Kallenberg, O. (1983)
{\em Random Measures.}
Academic Press. 
\bibitem{KAL1}
Kallenberg, O. (2002)
{\em Foundations of Modern Probability}.
Springer-Verlag.
\bibitem{KL}
Konstantopoulos, T. and Last, G. (2000)
On the dynamics and performance of stochastic fluid systems.
{\em J.\ Appl.\ Prob.} {\bf 37}, 652-667.
\bibitem{KZV}
Konstantopoulos, T., Zazanis, M.\  and de Veciana, G. (1997)
Conservation laws and reflection mappings with an application
to multiclass mean value analysis for stochastic fluid queues.
{\em Stoch.\ Proc.\ Appl.}  {\bf 65}, 139-146.
\bibitem{KoSa}
Kozlova, M.\ and Salminen, P. (2004)
Diffusion local time storage.
{\em Stoch.\ Proc.\ Appl.}  {\bf 114}, 211-229.
\bibitem{K}
Kyprianou, A. (2006)
{\em Introductory Lectures on Fluctuations of L\'evy Processes
with Applications.}
Springer.
\bibitem{KP}
Kyprianou, A.\ and Palmowski, Z. (2004)
A martingale review of some fluctuation theory for spectrally
negative L\'evy processes.
{\em Sem.\ Prob.} {\bf  28}, 16-29, Lecture Notes in Math., Springer.
\bibitem{MNS}
Mannersalo, P., Norros, I.\ and Salminen, P. (2004)
A storage process with local time input.
 {\em Queueing Syst. } {\bf 46}, 557-577.
\bibitem{PIT}
Pitman, J. (1986)
Stationary excursions.
{\em Sem.\ Prob.} {\bf  XXI}, 289-302, Lecture Notes in Math. 1247, Springer.
%editor= {{J. Az\'ema} and {P. A. Meyer} and {M. Yor}},
\bibitem{Sa}
Salminen, P. (1993)
On the distribution of diffusion local time.
 {\em Stat. Prob. Letters} {\bf 18}, 219-225. 
\bibitem{Si}
Sirvi\"o, M. (2006)
On an inverse subordinator storage.
{\em Helsinki University of Technology, Inst. Math. Report series} {\bf A501}.
\bibitem{TSIR}
Tsirelson, B. (2004)
Non-classical stochastic flows and continuous products.
{\em Probability Surveys} {\bf 1}, 173-298.
\bibitem{zolotarev64}
Zolotarev, V.M. (1964)
The first-passage time of a level and the behaviour at infinity for a
class of processes with independent increments.
{\em Theory Prob. Appl.} {\bf 9}, 653-664.


\end{thebibliography}
\end{document}